\documentclass[reqno,a4paper]{amsart}

\usepackage{amsmath,amssymb,amsfonts,amsthm,amscd}
\usepackage{geometry}
\usepackage{mathrsfs}
\usepackage{graphicx}
\usepackage{xcolor}
\usepackage{hyperref}
\usepackage[square,numbers]{natbib}
\hypersetup{
	colorlinks=true,
	linkcolor=blue,
	citecolor=blue,
	urlcolor=orange
}

\numberwithin{equation}{section}
\allowdisplaybreaks[3]

\newtheorem{theorem}{Theorem}[section]
\newtheorem{proposition}[theorem]{Proposition}
\newtheorem{lemma}[theorem]{Lemma}
\newtheorem{corollary}[theorem]{Corollary}

\theoremstyle{definition}
\newtheorem{definition}[theorem]{Definition}

\title{
The exterior Dirichlet problem for special Lagrangian equations
}

	\author{Yu Lei}
\address{School of Mathematics and Center for Nonlinear Studies, Northwest University, Xi'an, 710127, PR China}
\email{leiyu@stumail.nwu.edu.cn}

\author{Zhisu Li}
\address{School of Mathematics and Center for Nonlinear Studies, Northwest University, Xi'an, 710127, PR China}
\email{lizhisu@nwu.edu.cn}
\date{\today}



\begin{document}

\begin{abstract}
We establish the existence and uniqueness theorem
for the exterior Dirichlet problem for the special Lagrangian equation 
with prescribed asymptotic behavior at infinity,
in both the viscosity setting for all the phases
and classical setting for the critical and supercritical phases.
These results generalize previous work by the second author
by removing restrictive assumptions on the asymptotic matrix 
and improving the decay rate to the order $2-n$. 
We also solve the interior Dirichlet problem for the critical special Lagrangian equation
and, as applications, all the above-mentioned corresponding problems
for the three dimensional quadratic Hessian equation  without any admissibility condition.
\end{abstract}
\maketitle
	\setcounter{tocdepth}{1}
		\tableofcontents
			
	\section{Introduction}
	
Let $\Omega \subset \mathbb{R}^n$ be a bounded domain with $n \geq 3$. In this paper, we investigate the exterior Dirichlet problem for the special Lagrangian equation
\begin{equation}\label{eqsle}
	\begin{cases}
		\displaystyle \sum_{i=1}^n \arctan \lambda_i(D^2u) = \Theta, & x \in \mathbb{R}^n \setminus \overline{\Omega}, \\[1.2ex]
		u = \varphi, & x \in \partial \Omega,
	\end{cases}
\end{equation}
where the solution is required to possess the quadratic asymptotic 
\begin{equation}\label{2-n}
	u(x) = \frac{1}{2}x^T A x + b \cdot x + c + O(|x|^{2-n}), \qquad |x| \to \infty,
\end{equation}
where $\Theta$ is a constant satisfying
\begin{equation}
	-\frac{n\pi}{2} < \Theta < \frac{n\pi}{2},
\end{equation}
and  $\lambda_i(D^2u)$ denote the eigenvalues of the Hessian matrix $D^2u$. 

In view of the quadratic asymptotics \eqref{2-n}, it is natural to impose on the asymptotic matrix $A$ the structural condition
\begin{equation}
	A \in \mathcal{A}_\Theta := \left\{ A \in \mathrm{Sym}(n) : \sum_{i=1}^n \arctan \lambda_i(A) = \Theta \right\},
\end{equation}
with $b \in \mathbb{R}^n$ and $c \in \mathbb{R}$.

The study of exterior Dirichlet problems for fully nonlinear elliptic equations with prescribed quadratic asymptotics originates from the seminal work of Caffarelli and Li \cite{CL03} on the Monge--Amp\`ere equation. They proved that if $u$ is a convex viscosity solution of
$$
\det(D^2u) = 1
$$
outside a bounded domain, then $u$ must have the asymptotic expansion
$$
u(x) = \frac{1}{2}x^T A x + b \cdot x + c + O(|x|^{2-n}), \qquad |x| \to \infty,
$$
for some $A \in \mathrm{Sym}(n)$, $b \in \mathbb{R}^n$, $c \in \mathbb{R}$. Moreover, they established existence and uniqueness of solutions to the exterior Dirichlet problem with prescribed such asymptotics.

For Hessian equations $\sigma_k(D^2u) = 1$, the exterior problem was first studied by Dai and Bao \cite{DB11, Dai11} in the case where the asymptotic matrix is a multiple of the identity. Subsequently, Bao, Li, and Li \cite{BLL14} introduced generalized symmetric functions and treated the case of positive definite asymptotic matrices $A > 0$. Their work initiated a systematic study of exterior problems for Hessian-type equations. Later developments have included  extensions to Hessian quotient equations \cite{LL18, DBL23}, special Lagrangian equations \cite{Li19}, and other fully nonlinear elliptic equations \cite{JLL21, LC24}.

Recently, two significant advances have been made in the Hessian equation theory. Bao and Jiang \cite{BJ26} studied the exterior Dirichlet problem for Hessian equations 
with $k$-admissible asymptotic matrices, allowing the asymptotic matrix to be merely $k$-admissible rather than positive definite. They introduced the key idea that the metric at infinity should be governed by the linearized operator at the asymptotic matrix, rather than by the asymptotic matrix itself. In a complementary direction, Li and Xiao \cite{LX26} proved the existence of smooth solutions to the exterior Hessian problem on non-convex rings, assuming only that the domain is strictly $k-1$ convex and star-shaped, thereby relaxing the strict convexity assumption that was essential in earlier works.

The special Lagrangian equation
\begin{equation}
	\sum_{i=1}^n \arctan \lambda_i(D^2u) = \Theta, \label{eq:SLE}
\end{equation}
was introduced by Harvey and Lawson \cite{HL82} in 1982 in their foundational work on calibrated geometries. A Lagrangian graph $(x, Du) \subset \mathbb{R}^n \times \mathbb{R}^n$ is called special if and only if it is a minimal submanifold, and \eqref{eq:SLE} is precisely the governing equation for such graphs.

The equation has been extensively studied. Yuan proved a Bernstein-type result for global semiconvex solutions \cite{Yua02}, and later established the convexity of the level sets for the critical and supercritical phases \cite{Y06}. Interior Hessian estimates were first obtained by Warren and Yuan \cite{WY09} for convex solutions, and subsequently extended by Warren and Yuan \cite{WY10} and Wang and Yuan \cite{WY14} to the critical and supercritical phases without requiring convexity.

For the Dirichlet problem on bounded domains, Caffarelli, Nirenberg, and Spruck \cite{CNS85} solved \eqref{eq:SLE} for $\Theta = (n-2)\pi/2$ when $n$ is even, and for $\Theta = (n-1)\pi/2$ when $n$ is odd, under suitable geometric conditions on the domain. In the supercritical regime $\Theta > (n-2)\pi/2$, Collins, Picard, and Wu \cite{CPW17} established the existence and uniqueness of smooth solutions, relying crucially on the concavity of the Lagrangian phase operator. The existence and uniqueness of viscosity solutions for the Dirichlet problem in the general case was later established by Harvey and Lawson \cite{HL09}.

For the Dirichlet problem on bounded domains, the special Lagrangian equation is now known to admit a unique viscosity solution for all phases, thanks to the works of Harvey and Lawson \cite{HL09} and Bhattacharya \cite{B24}. In the supercritical regime $\Theta > (n-2)\pi/2$, smooth solutions were established by Collins, Picard, and Wu \cite{CPW17}. The critical case $\Theta = (n-2)\pi/2$ on bounded domains remained open.

A natural question is whether the viscosity theory of Harvey--Lawson \cite{HL09} for all phases and the smooth theory of Collins--Picard--Wu \cite{CPW17} for the supercritical regime can be extended to the exterior Dirichlet problem with prescribed quadratic asymptotics, especially in the critical case.

The present work draws inspiration from the recent contributions of Bao and Jiang \cite{BJ26} on exterior Dirichlet problems for Hessian equations, from the results of Li and Xiao \cite{LX26} on exterior Hessian problems on non-convex rings, and from the work of Fu, Yau, and Zhang \cite{FYZ26} on the critical LYZ equation in K\"{a}hler geometry.

For the special Lagrangian equation, the exterior Dirichlet problem was first solved by Li \cite{Li19}, under certain restrictive assumptions on the asymptotic matrix. The present paper continues this line of research. Our main contributions are threefold. First, we remove the restrictive assumptions on the asymptotic matrix $A$, allowing it to lie anywhere on the phase level set $\mathcal{A}_\Theta$ without any positive definiteness or admissibility condition. Second, we improve the decay rate to the order $2-n$, which is the natural decay rate for the fundamental solution of the linearized operator. Third, we solve the Dirichlet problem in the critical case $\Theta = (n-2)\pi/2$, which was not covered in previous works. As applications, our results solve the $3$-dimensional $2$-Hessian equation $\sigma_2 = 1$ without imposing any admissibility condition.

\begin{theorem}[Viscosity exterior solution]
	\label{thm:viscosity}
	Let $n \geq 3$. Suppose that $\Omega \subset \mathbb{R}^n$ is a bounded uniformly convex $C^{1,1}$ domain, and $\varphi \in C^0(\partial \Omega)$ is uniformly semiconvex with respect to $\partial \Omega$. Then, for any given $A \in \mathcal{A}_\Theta$ with
	$$
	-\frac{n\pi}{2} < \Theta < \frac{n\pi}{2},
	$$
	and  given $b \in \mathbb{R}^n$, there exists
	$c_* = c_*(n,\Theta,A,b,\Omega,\varphi) > 0$
	such that for every $c > c_*$, the exterior Dirichlet problem \eqref{eqsle}
	admits a unique viscosity solution $u \in C^0(\mathbb{R}^n\setminus\Omega)$ satisfying \eqref{2-n}.
\end{theorem}

\begin{theorem}[Smooth exterior solutions]
	\label{thm:smooth}
	Let $n \geq 3$. Suppose that $\Omega\subset \mathbb{R}^n$ is a smooth bounded domain with strictly convex boundary, and $\varphi \in C^\infty(\partial \Omega)$. Then, for any given $A \in \mathcal{A}_\Theta$ with
	$$
	\frac{(n-2)\pi}{2} \leq \Theta < \frac{n\pi}{2},
	$$
	and any given $b \in \mathbb{R}^n$, there exists a constant
	$c_* = c_*(n,\Theta,A,b,\Omega,\|\varphi\|_{C^2(\partial \Omega)}) > 0$
	such that for every $c > c_*$, the exterior Dirichlet problem \eqref{eqsle} admits a unique smooth solution $u \in C^\infty(\mathbb{R}^n\setminus\Omega)$ satisfying \eqref{2-n}.
\end{theorem}

\section{Preliminaries}
\subsection{Notation and definitions}
The following notation and definitions will be used throughout this paper. Since the special Lagrangian equation depends only on the Hessian of $u$, the linear term $b \cdot x$ in the quadratic asymptotics can be absorbed into the boundary data by replacing $\varphi$ with $\varphi - b \cdot x$. Without loss of generality,  we assume $b = 0$. 
	\begin{equation*}
	Q_c(x) := \frac{1}{2}x^T A x + c.
	\label{def:Qc}
\end{equation*}
We denote by $\mathrm{Sym}(n)$ the space of $n \times n$ real symmetric matrices. For $M \in \mathrm{Sym}(n)$, we write $\lambda(M) = (\lambda_1(M), \ldots, \lambda_n(M))$ for its eigenvalues.

Define the Lagrangian phase function
\begin{equation*}
	F(M) := \sum_{i=1}^n \arctan \lambda_i(M), \qquad   F(A)=\Theta.
	\label{def:H}
\end{equation*}
For $M = (M_{ij}) \in \mathrm{Sym}(n)$, denote the matrix $G$ by
\begin{equation}
G_{ij} := \frac{\partial F}{\partial M_{ij}}(A), \qquad	G= (G_{ij})_{i,j=1}^n      .
	\label{def:G}
\end{equation}
$G$ is a positive definite matrix for every $A \in \mathcal{A}_\Theta$ since 
$$
G \sim \mathrm{diag}\left( \frac{1}{1+\lambda_1(A)^2}, \dots, \frac{1}{1+\lambda_n(A)^2} \right).
$$
We define the ellipsoid norm in terms of $G$ by
\begin{equation}
	\rho(x) := (x^T G^{-1} x)^{1/2}.
	\label{def:rho}
\end{equation}
In this setting, under the linear change of variables $x = G^{1/2}y$, set $w(y) = v(G^{1/2}y)$. Then the linearized operator
\begin{equation}
	L_A v := \sum_{i,j=1}^n G_{ij} v_{ij},
	\label{def:LA}
\end{equation}
satisfies $L_A v(x) = \Delta_y w(y)$, and the norm $\rho(x)$ becomes $|y|$. 
For $R > 0$, set
\begin{equation}
	E_R := \{ x \in \mathbb{R}^n : \rho(x) < R \}.
\end{equation}
Since $G > 0$, each $E_R$ is a smooth strictly convex ellipsoid.

\subsection{Some preliminary lemmas}
In this subsection, we collect some well-known preliminary lemmas which will be used in the proof of the main theorems.

\begin{lemma}[\cite{Li19}, Lemma 2.4]
	\label{lem:comparison}
	Let $\Omega \subset \mathbb{R}^n$ be a domain, and let $\underline{u}, \overline{u} \in C^0(\overline{\Omega})$ be a viscosity subsolution and a  viscosity supersolution, respectively, of the special Lagrangian equation
	$$
	\sum_{i=1}^n \arctan \lambda_i(D^2u) = \Theta,
	$$
	with $\Theta \in (-n\pi/2, n\pi/2)$. Suppose $\underline{u} \le \overline{u}$ on $\partial \Omega$, and additionally,
	$$
	\lim_{|x| \to +\infty} (\underline{u} - \overline{u})(x) = 0,
	$$
	provided $\Omega$ is unbounded. Then $\underline{u} \le \overline{u}$ in $\Omega$.
\end{lemma}

\begin{lemma}[\cite{Li19}, Lemma 2.5]
	\label{lem:perron}
	Let $\Omega \subset \mathbb{R}^n$ be a domain, and let $\underline{u}, \overline{u} \in C^0(\overline{\Omega})$ be  viscosity subsolution and a  viscosity supersolution, respectively, of the special Lagrangian equation
	$$
	\sum_{i=1}^n \arctan \lambda_i(D^2u) = \Theta,
	$$
	with $\Theta \in (-n\pi/2, n\pi/2)$. Suppose $\underline{u} \le \overline{u}$ in $\Omega$, $\underline{u} =\varphi$ on $\partial \Omega$, and additionally,
	$$
	\lim_{|x| \to +\infty} (\overline{u} - \underline{u})(x) = 0,
	$$
	provided $\Omega$ is unbounded. Then there exists a viscosity solution $u \in C^0(\overline{\Omega})$ of
	$$
	\sum_{i=1}^n \arctan \lambda_i(D^2u) = \Theta \quad \text{in } \Omega,
	$$
	such that
	$$
	\underline{u} \le u \le \overline{u} \quad \text{in } \Omega,
	$$
	and $u = \underline{u} = \varphi$ on $\partial \Omega$.
\end{lemma}

\begin{lemma}[\cite{Y06}, Lemma 2.1]
	\label{lem:convex}
	For $\Theta \in [(n-2)\pi/2, n\pi/2)$, the level set
	\begin{equation*}
		L_\Theta := \left\{ \lambda \in \mathbb{R}^n : \sum_{i=1}^n \arctan \lambda_i \geq \Theta \right\}
		\label{def:LTheta}
	\end{equation*}
	is convex.
\end{lemma}

\begin{lemma}[Far-field barriers]
	\label{lem:farfield}
	Assume $A \in \mathcal{A}_\Theta$ with $-n\pi/2< \Theta < n\pi/2$. For
	$B>0$ and $\delta>0$, there exists a structural constant $R_* > 0$, depending only on $n$, $\Theta$, $A$, $B$, and $\delta$, such that for every $R > R_*$, the function
	\begin{equation}
		u_\infty^-(x) := Q_c(x) - B \rho(x)^{2-n} + \delta \rho(x)^{-n}, \qquad \rho(x) \ge R,
		\label{farfield-sub}
	\end{equation}
	is a smooth subsolution of the special Lagrangian equation
	\begin{equation}
		\sum_{i=1}^n \arctan \lambda_i(D^2u) = \Theta
	\end{equation}
	in $\mathbb{R}^n \setminus E_R$, and satisfies
	\begin{equation}
		Q_c(x) - C|x|^{2-n} < u_\infty^-(x) < Q_c(x), \qquad \rho(x) > R,
		\label{farfield-estimate}
	\end{equation}
	for some constant $C = C(n,\Theta,A,B) > 0$.
	
	Moreover, the function
	\begin{equation*}
		u_\infty^+(x) := Q_c(x) + B \rho(x)^{2-n} - \delta \rho(x)^{-n}, \qquad \rho(x) \ge R,
		\label{farfield-super}
	\end{equation*}
	is a smooth supersolution in $\mathbb{R}^n \setminus E_R$.
\end{lemma}

\begin{proof}
	For every real $\alpha$, differentiating $\rho^\alpha = (x^T G^{-1}x)^{\alpha/2}$ gives
	\begin{equation*}
		\partial_{ij}\rho^\alpha = \alpha \rho^{\alpha-2}(G^{-1})_{ij} + \alpha(\alpha-2)\rho^{\alpha-4}(G^{-1}x)_i(G^{-1}x)_j.
		\label{eq:rho-derivative}
	\end{equation*}
	Since $|G^{-1}x| \le C\rho$, it follows that
	\begin{equation*}
		|\partial_{ij}\rho^\alpha| \le C\rho^{\alpha-2},
	\end{equation*}
	where $C$ depends only on $n$, $\alpha$, and the eigenvalue bounds of $A$.
	
	With the change of variables $y = G^{-1/2}x$, one has $\rho(x) = |y|$ and $L_A = \Delta_y$. Hence
	\begin{equation*}
		L_A(\rho^{2-n}) = 0, \qquad L_A(\rho^{-n}) = 2n\rho^{-n-2}.
		\label{eq:LA-rho}
	\end{equation*}
	Set
	\begin{equation*}
		\psi(x) := -B \rho(x)^{2-n} + \delta \rho(x)^{-n}.
	\end{equation*}
	Then $D^2u_\infty^- = A + D^2\psi$ and
	\begin{equation}
		L_A\psi = 2n\delta \rho^{-n-2}.
		\label{eq:LA-psi}
	\end{equation}
	On the other hand,
	\begin{equation}
		|D^2\psi| \le C_1 B \rho^{-n} + C_2\delta \rho^{-n-2}.
		\label{eq:D2psi-estimate}
	\end{equation}
	
	Since $F(M) = \sum_{i=1}^n \arctan \lambda_i(M)$ is smooth near $A$, there exists $\eta > 0$ such that whenever $H \in \mathrm{Sym}(n)$ and $|H| \le \eta$,
	\begin{equation}
		F(A + H) = F(A) + L_A H + O(|H|^2).
		\label{taylor}
	\end{equation}
 Choosing $R \ge R_*$ large enough and using \eqref{eq:D2psi-estimate}, we obtain $|D^2\psi|\leq \eta$. 
	
	Taking $H = D^2\psi$, for large $R_*$ so that $|D^2\psi| \le \eta$ for $\rho \ge R$, we obtain from \eqref{eq:LA-psi} and \eqref{eq:D2psi-estimate},
	\begin{equation}
		F(A + D^2\psi) = F(A) + L_A(D^2\psi) + O(|D^2\psi|^2)= \Theta + 2n\delta \rho^{-n-2} + O(\rho^{-2n}).
	\end{equation}
	Indeed, there exists a constant $C > 0$ such that
	$$
	|O(\rho^{-2n})| \le C\rho^{-2n}.
	$$
	Thus
	$$
	F(A + D^2\psi) \ge \Theta + 2n\delta \rho^{-n-2} - C\rho^{-2n}.
	$$
	To ensure the right-hand side is at least $\Theta$, it suffices to have
	$$
	2n\delta \rho^{-n-2} - C\rho^{-2n} \ge 0,
	$$
	which is equivalent to
	$$
	2n\delta \rho^{n-2} \ge C.
	$$
	Since $n \ge 3$, we may choose $R$ sufficiently large such that
	$$
	R^{n-2} \ge \frac{C}{2n\delta}.
	$$
	Then for all $\rho \ge R$, we have $2n\delta \rho^{n-2} \ge C$, and consequently
	$$
	2n\delta \rho^{-n-2} - C\rho^{-2n} \ge 0.
	$$
	Therefore,
	$$
	F(A + D^2\psi) \ge \Theta \quad \text{for } \rho \ge R.
	$$
	Finally, since $\rho$ and $|x|$ are comparable, for $R$ sufficiently large,
	\begin{equation}
		-C|x|^{2-n} < -B \rho^{2-n} + \delta \rho^{-n} < 0.
	\end{equation}
	This gives the two-sided estimate
	\begin{equation}
		Q_c(x) - C|x|^{2-n} < u_\infty^-(x) < Q_c(x), \qquad \rho(x) \ge R.
	\end{equation}
	The supersolution is proved in the same way, with
	\begin{equation}
		\psi_+(x) := B \rho^{2-n} - \delta \rho^{-n}.
	\end{equation}
	The linear term changes sign, and the same Taylor expansion shows that, for $\rho$ sufficiently large, the negative linear contribution dominates the quadratic remainder; hence $F(A + D^2\psi_+) \le \Theta$, with increasing $R$ if necessary.
\end{proof}

\begin{lemma}[\cite{BJ26}, Lemma 2.4]
	\label{lem:regmax}
	Let $\mu > 0$. There exists a function  $M_\mu \in C^\infty(\mathbb{R}^2)$ with the following properties:
	\begin{enumerate}
		\item[(i)] $M_\mu$ is convex and nondecreasing in each variable;
		\item[(ii)] $M_\mu(s,t) = \max\{s,t\}$ whenever $|s-t| \ge \mu$;
		\item[(iii)] $\max\{s,t\} \le M_\mu(s,t) \le \max\{s,t\} + \mu$.
	\end{enumerate}
\end{lemma}

	The function $M_\mu$ is constructed as follows \cite{BJ26}. Let $\eta \in C_c^\infty((-1,1))$ be an even nonnegative function with $\int_{\mathbb{R}} \eta(r)\,dr = 1$, and define
	$$
	\eta_\mu(r) := \mu^{-1} \eta(r/\mu), \qquad 
	\chi_\mu(r) := \int_{\mathbb{R}} |r - s| \eta_\mu(s)\,ds.
	$$
	Then $\chi_\mu \in C^\infty(\mathbb{R})$ is even and convex, satisfies $\chi_\mu(r) = |r|$ for $|r| \ge \mu$, and $|r| \le \chi_\mu(r) \le |r| + \mu$. Finally, set
	$$
	M_\mu(s,t) := \frac{s + t + \chi_\mu(s - t)}{2}.
	$$

\begin{lemma}
	\label{max-sub}
	Let $M_\mu$ be as in Lemma \ref{lem:regmax} and $\Theta\geq (n-2)\pi/2$. Suppose $u_1, u_2 \in C^2(\mathbb{R}^n)$ satisfy 
	$$
	F(D^2u_1) \ge \Theta, \qquad F(D^2u_2) \ge \Theta \quad \text{in } \quad  \mathbb{R}^n.  
	$$
   Then $M_\mu(u_1, u_2) \in C^2(\mathbb{R}^n)$ and
	$$
	F(D^2 M_\mu(u_1, u_2)) \ge \Theta \quad \text{in } \quad \mathbb{R}^n.
	$$
	
\end{lemma}

\begin{proof}
	By Lemma \ref{lem:regmax}, it is clear that $M_\mu(u_1, u_2) \in C^2(\mathbb{R}^n)$. Let $r = u_1 - u_2$ and put
	$$
	\alpha := \frac{1 + \chi_\mu'(r)}{2}, \qquad \gamma := \frac{1}{2} \chi_\mu''(r) \ge 0.
	$$
	Then $0 \le \alpha \le 1$ and direct differentiation gives
	\begin{equation}
		D^2 M_\mu(u_1, u_2) = \alpha D^2 u_1 + (1 - \alpha) D^2 u_2 + \gamma D(u_1 - u_2) \otimes D(u_1 - u_2).
		\label{eq:max-hessian}
	\end{equation}
	Since 
	$\left\{ M \in \mathrm{Sym}(n) :F(M) \ge \Theta \right\}$
	  is convex, we have
	$$
	F(\alpha D^2u_1 + (1-\alpha)D^2u_2) \ge \Theta.
	$$
	Moreover, since $\gamma D(u_1 - u_2) \otimes D(u_1 - u_2) \ge 0$  and $F$ is monotone, we obtain
	$$
	F(D^2 M_\mu(u_1, u_2)) \ge F(\alpha D^2u_1 + (1-\alpha)D^2u_2) \ge \Theta.
	$$

\end{proof}
\section{Viscosity exterior solutions}
We define
$$
M_\varepsilon := A + \varepsilon G^{-1}.
$$
And from \eqref{taylor}, there exists $\varepsilon_0>0$ sufficiently small such that for every $0<\varepsilon<\varepsilon_0$,  $$
F(M_\varepsilon) = F(A) + \varepsilon L_A(G^{-1}) + O(\varepsilon^2)
= \Theta + n\varepsilon + O(\varepsilon^2)>\Theta.
$$

\begin{definition}[Uniform convexity]
	\label{def:convex}
	Let $\Omega \subset \mathbb{R}^n$ be a bounded domain. The domain $\Omega$ is said to be \emph{uniformly convex in the supporting-plane sense} if there exist constants $r_\Omega, c_\Omega, \kappa_\Omega, \xi_\Omega > 0$ such that, for every $\xi, x \in \partial \Omega$,
	\begin{equation}\label{unicon}
		-c_\Omega |x-\xi|^2 \le (x-\xi)\cdot \nu(\xi) \le -\kappa_\Omega |x-\xi|^2, \quad |x-\xi| < r_\Omega, 
	\end{equation}
	and
	\begin{equation}\label{unicon2}
		(x-\xi)\cdot \nu(\xi) \le -\xi_\Omega, \quad |x-\xi| \ge r_\Omega,
	\end{equation}
	where $\nu(\xi)$ denotes the exterior unit normal to $\partial \Omega$ at $\xi$.
\end{definition}

\begin{definition}[Uniform semiconvexity]
	\label{def:semiconvex}
	Let $\Omega \subset \mathbb{R}^n$ be a bounded $C^{1,1}$ domain. A function $\varphi \in C^0(\partial \Omega)$ is called \emph{uniformly semiconvex with respect to $\partial \Omega$} if there exist constants $r_\varphi, K_\varphi, P_\varphi > 0$ such that, for every $\xi \in \partial \Omega$, there exists a vector
	$$
	\tau_\xi \in T_\xi \partial \Omega, \qquad |\tau_\xi| \le P_\varphi,
	$$
	depending on $\varphi$, such that
	\begin{equation}\label{semi}
		\varphi(x) \ge \varphi(\xi) + \tau_\xi \cdot (x-\xi) - K_\varphi |x-\xi|^2 
	\end{equation}
	whenever $x \in \partial \Omega$ and $|x-\xi| < r_\varphi$.
\end{definition}

\begin{lemma}[Quadratic boundary supports]
	\label{lem:boundary-supports}
	Let \(\Omega\) and \(\varphi\) satisfy the hypotheses of Theorem \ref{thm:viscosity}. 
	For every \(0 < \varepsilon \le \varepsilon_0\) and every \(\xi \in \partial \Omega\), 
	there exists a uniformly bounded vector \(p_\xi^\varepsilon\), 
	depending only on \(n, \Theta, A, \Omega, K_\varphi, P_\varphi\), such that
	$$
	\omega_\xi^\varepsilon(x) := \varphi(\xi) + p_\xi^\varepsilon \cdot (x-\xi) 
	+ \frac{1}{2}(x-\xi)^T M_\varepsilon (x-\xi), \qquad x \in \mathbb{R}^n,
	$$
	satisfies
	$$
	\omega_\xi^\varepsilon(\xi) = \varphi(\xi), \qquad 
	\omega_\xi^\varepsilon < \varphi \quad \text{on } \partial \Omega \setminus \{\xi\}.
	$$
	Moreover, each \(\omega_\xi^\varepsilon\) is a strict smooth subsolution of the 
	special Lagrangian equation, since
	$$
	F(M_\varepsilon) = F(A + \varepsilon G^{-1}) 
	= \Theta + n\varepsilon + O(\varepsilon^2) > \Theta
	$$
	for all sufficiently small \(\varepsilon > 0\).
\end{lemma}

\begin{proof}
	The number $\varepsilon_0 > 0$ has been fixed so that $\|M_\varepsilon\| \le c(n,\Theta,A)$ for $0 < \varepsilon \le \varepsilon_0$.
	
	Choose
	$$
	p_\xi^\varepsilon := \tau_\xi + N\nu(\xi),
	$$
	where $N > 0$ will be chosen later.
	
	If $0 < |x-\xi| < r_0 := \min\{r_\Omega, r_\varphi\}$, then \eqref{unicon} and \eqref{semi} give
	$$
	\begin{aligned}
		\omega_\xi^\varepsilon(x) - \varphi(x)
		&\le N\nu(\xi)\cdot(x-\xi) + K_\varphi |x-\xi|^2 
		+ \frac{1}{2}\|M_\varepsilon\|\,|x-\xi|^2 \\
		&\le -\left(N\kappa_\Omega - K_\varphi - \frac{1}{2}\|M_\varepsilon\|\right)|x-\xi|^2.
	\end{aligned}
	$$
	Taking $N = N(n,\Theta,A,\Omega,K_\varphi)$ sufficiently large makes this negative 
	for every $0 < |x-\xi| < r_0$.
	
	It remains to consider the compact part $\{x\in \partial \Omega : |x-\xi| \ge r_0\}$. 
	Since $\varphi$ is continuous and $\partial \Omega$ is compact, there exists a constant 
	$C_\partial$, independent of $\xi$ and $\varepsilon$, such that
	$$
	\tau_\xi\cdot(x-\xi) + \frac{1}{2}(x-\xi)^T M_\varepsilon (x-\xi) 
	- \left(\varphi(x)-\varphi(\xi)\right) \le C_\partial
	$$
	on this set. The global separation in Definition \ref{def:convex} then yields
	\[
	\omega_\xi^\varepsilon(x) - \varphi(x) \le -N\xi_\Omega + C_\partial.
	\]
	Increasing $N$ once more, still only in terms of the fixed data, makes this strictly negative. 
	Hence $\omega_\xi^\varepsilon$ touches $\varphi$ from below at $\xi$ and is strictly below 
	$\varphi$ at every other boundary point.
	
	Finally, $D^2\omega_\xi^\varepsilon = M_\varepsilon$, and the strict smooth subsolution 
	property follows from
	$$
	F(M_\varepsilon) = F(A + \varepsilon G^{-1}) 
	= \Theta + n\varepsilon + O(\varepsilon^2) > \Theta
	$$
	for all sufficiently small $\varepsilon > 0$.
\end{proof}

Define
$$
w_\varepsilon(x) := \sup_{\xi \in \partial \Omega} \omega_\xi^\varepsilon(x), \qquad x \in \mathbb{R}^n.
$$
Since the family $\{\omega_\xi^\varepsilon\}_{\xi \in \partial \Omega}$ is locally uniformly Lipschitz in $x$, with constants independent of $\xi$, the function $w_\varepsilon$ is continuous. By the stability of viscosity subsolutions under locally bounded suprema (see  \cite{89I}, Proposition 2.2), $w_\varepsilon$ is a viscosity subsolution of the special Lagrangian equation. The touching property of the family $\{\omega_\xi^\varepsilon\}$ gives
$$
w_\varepsilon = \varphi \quad \text{on } \partial \Omega.
$$
A direct computation gives
$$
\begin{aligned}
	\omega_\xi^\varepsilon(x) - Q_0(x)
	&= \frac{1}{2}x^T(M_\varepsilon - A)x + (p_\xi^\varepsilon - M_\varepsilon \xi - b)\cdot x + \varphi(\xi) - p_\xi^\varepsilon \cdot \xi + \frac{1}{2}\xi^T M_\varepsilon \xi.
\end{aligned}
$$
Since \(\partial \Omega\) is compact and the vectors \(p_\xi^\varepsilon\) are bounded uniformly for \(0<\varepsilon\le\varepsilon_0\), there exists a constant \(C>0\), independent of \(\xi\) and \(\varepsilon\), such that
$$
|p_\xi^\varepsilon - M_\varepsilon \xi - b| \le C, \qquad 
\left|\varphi(\xi) - p_\xi^\varepsilon \cdot \xi + \frac{1}{2}\xi^T M_\varepsilon \xi\right| \le C, \qquad \xi \in \partial \Omega.
$$
Consequently, using the relation \(|x| \le C(A)\rho(x)\), we obtain
$$
\left|(p_\xi^\varepsilon - M_\varepsilon \xi - b)\cdot x + \varphi(\xi) - p_\xi^\varepsilon \cdot \xi + \frac{1}{2}\xi^T M_\varepsilon \xi\right|
\le C(1+|x|) \le C(1+\rho(x)).
$$
On the other hand,
$$
\frac{1}{2}x^T(M_\varepsilon - A)x = \frac{\varepsilon}{2}x^T G^{-1}x = \frac{\varepsilon}{2}\rho(x)^2.
$$
Thus, uniformly in \(\xi \in \partial \Omega\),
$$
\omega_\xi^\varepsilon(x) - Q_0(x) \le C(1+\rho(x)) + \frac{\varepsilon}{2}\rho(x)^2.
$$
Taking the supremum over $\xi \in \partial \Omega$ yields the upper bound for $w_\varepsilon - Q_0$.

For the reverse inequality, fix some $\xi_0 \in \partial \Omega$. Since \(w_\varepsilon \ge \omega_{\xi_0}^\varepsilon\), the preceding expansion implies
$$
w_\varepsilon(x) - Q_0(x) \ge \omega_{\xi_0}^\varepsilon(x) - Q_0(x) 
\ge -C(1+\rho(x)) + \frac{\varepsilon}{2}\rho(x)^2.
$$
Combining the two estimates proves
\begin{equation}\label{sup}
	|w_\varepsilon(x) - Q_0(x)| \le C_0(1+\rho(x)) + \frac{\varepsilon}{2}\rho(x)^2, \qquad x \in \mathbb{R}^n.
\end{equation}

\begin{proposition}[Viscosity global subsolution]
	\label{prop:global-subsolution}
	Under the assumptions of Theorem \ref{thm:viscosity}, for any given $b \in \mathbb{R}^n$, one can find a constant $c_* > 0$ with the following property: for every $c > c_*$, there exists a function $\underline{u}_c \in C(\overline{\Omega})$ which is a viscosity subsolution of the special Lagrangian equation, satisfies
	$$
	\underline{u}_c = \varphi \quad \text{on } \partial \Omega, \qquad \underline{u}_c \le Q_c,
	$$
	and possesses the asymptotic expansion
	$$
	\underline{u}_c(x) = Q_c(x) + O(|x|^{2-n}) \quad \text{as } |x| \to \infty.
	$$
	Furthermore, the following quantitative bound holds:
	$$
	Q_c(x) - C c^{n-1} |x|^{2-n} \le \underline{u}_c(x) \le Q_c(x), \qquad |x| \ge c,
	$$
	where the constant $C$  is independent of $c$.
\end{proposition}
\begin{proof}
	Let
	$$
	m := \min_{\partial E_R} (w_\varepsilon - Q_0), \qquad 
	M := \max_{\partial E_{2R}} (w_\varepsilon - Q_0).
	$$
	Set $\varepsilon = R^{-n}$. Then from \eqref{sup}, there exists a constant $K_0 > 0$ such that
	$$
	|m| + |M| \le K_0(1+R).
	$$
	Choose $\tau > 0$ sufficiently small (to be determined below), and choose $\Lambda > 0$ sufficiently large (also to be determined below).
	For $c > 0$, set
	$$
	R_c := \tau c, \qquad \varepsilon_c := R_c^{-n}, \qquad \delta_c := \Lambda c^2 R_c^{n-2}.
	$$
	Write $w_c := w_{\varepsilon_c}$, and set
	$$
	m_c := \min_{\partial E_{R_c}} (w_c - Q_0), \qquad 
	M_c := \max_{\partial E_{2R_c}} (w_c - Q_0).
	$$
	Define
	$$
	B_c^- := (c + \delta_c R_c^{-n} - m_c) R_c^{n-2}, \qquad
	B_c^+ := (c + \delta_c (2R_c)^{-n} - M_c) (2R_c)^{n-2}.
	$$
	A direct computation gives
	$$
	B_c^+ - B_c^- 
	= R_c^{n-2} \left[ (2^{n-2} - 1)c + m_c - 2^{n-2}M_c + (2^{-2} - 1)\delta_c R_c^{-n} \right].
	$$
	 And $R_c = \tau c$, we get
	$$
	B_c^+ - B_c^- 
	\ge R_c^{n-2} \left[ (2^{n-2} - 1)c - K_0(1+\tau c) - C\Lambda \tau^{-2} \right].
	$$
	Choose $\tau > 0$ so small that $(2^{n-2} - 1) - K_0\tau > 0$. Then, after choosing $\Lambda > 0$ and then taking $c$ sufficiently large (say $c \ge c_*$). Hence $B_c^- < B_c^+$ for all $c \ge c_*$.
	Thus there exists $c_* > 0$ such that for every $c \ge c_*$, the interval $[B_c^-, B_c^+]$ is nonempty.
	Fix any $B_c \in [B_c^-, B_c^+]$. Define the far-field subsolution
	$$
	u_\infty^c(x) := Q_c(x) - B_c \rho(x)^{2-n} + \delta_c \rho(x)^{-n}, \qquad \rho(x) \ge R_c.
	$$
	By Lemma \ref{lem:farfield}, for all sufficiently large $c$, $u_\infty^c$ is a smooth subsolution of the special Lagrangian equation in $\mathbb{R}^n \setminus E_{R_c}$. The choice of $B_c$ gives the interface inequalities
	\begin{align}
		u_\infty^c &\le w_c \quad \text{on } \partial E_{R_c}, \label{eq:interface1} \\
		u_\infty^c &\ge w_c \quad \text{on } \partial E_{2R_c}. \label{eq:interface2}
	\end{align}
	Define
	\begin{equation}
		\underline{u}_c(x) :=
		\begin{cases}
			w_c(x), & x \in E_{R_c} \setminus \overline{\Omega}, \\[1.2ex]
			\max\{w_c(x), u_\infty^c(x)\}, & x \in E_{2R_c} \setminus E_{R_c}, \\[1.2ex]
			u_\infty^c(x), & x \in \mathbb{R}^n \setminus E_{2R_c}.
		\end{cases} 
	\end{equation}
	Then we have:
	
	1. $\underline{u}_c$ is continuous and satisfies
	\[
	\sum_{i=1}^n \arctan \lambda_i(D^2\underline{u}_c) \ge \Theta
	\]
	in $\mathbb{R}^n \setminus \overline{\Omega}$ in the viscosity sense, by the maximum stability of viscosity subsolutions.
	
	2. $\underline{u}_c = \varphi$ on $\partial \Omega$, since $\underline{u}_c = w_c = \varphi$ on $\partial \Omega$.
	
	3. For $|x|$ sufficiently large,
	$$
	\underline{u}_c(x) = u_\infty^c(x) = \frac{1}{2}x^T A x + c + O(|x|^{2-n}) \quad (|x| \to \infty).
	$$
	
	Let $$\overline{u}:= Q_c = \frac{1}{2}x^T A x + c.$$
	To apply the Perron method (Lemma \ref{lem:perron}), we need to verify that 
	$\overline{u}$ is a viscosity supersolution and that $\underline{u} \le \overline{u}\) in \(\Omega$. It remains to ensure that $w_c \le Q_c$ in the bounded part. For $x \in E_{2R_c} \setminus \overline{\Omega}$, estimate \eqref{sup} gives
	$$
	w_c(x) - Q_c(x) \le C(1+\rho(x)) + \frac{\varepsilon_c}{2}\rho(x)^2 - c.
	$$
	Since $\rho(x) < 2R_c$, $\varepsilon_c = R_c^{-n}$, and $R_c = \tau c$, we have
	$$
	w_c(x) - Q_c(x) \le C_0 + C_1 \tau c - c.
	$$
	We choose $\tau > 0$ at the beginning so small that $C_1\tau \le 1/2$. 
	After increasing $c_*$ such that $C_0 \le c/2$, the last quantity is nonpositive. 
	Hence $w_c \le Q_c$ in $E_{2R_c} \setminus \overline{\Omega}$. 
	Combined with the far-field inequality $u_\infty^c \le Q_c$, we obtain $\underline{u}_c \le Q_c$ in $\Omega$.
	Recall that $G > 0$, and let
	$$
	\alpha_G := \lambda_{\min}(G^{-1})^{1/2} = \lambda_{\max}(G)^{-1/2},
	$$
	then $\rho(x) \ge \alpha_G |x|$. We choose $\tau_0 > 0$ so small that $\tau_0 \le \alpha_G/2$. Then, for every $0 < \tau \le \tau_0$ and $R_c = \tau c$,
	$$
	|x| \ge c \quad \Longrightarrow \quad \rho(x) \ge 2\tau |x| \ge 2R_c.
	$$
	Hence, for $|x| \ge c$, we have $\underline{u}_c = u_\infty^c$. By Lemma \ref{lem:farfield}, we obtain
	$$
	Q_c(x) \ge \underline{u}_c(x) \ge Q_c(x) - C c^{n-1} |x|^{2-n}, \qquad |x| \ge c,
	$$
	and
	$$
	\underline{u}_c(x) = Q_c(x) + O(|x|^{2-n}) \quad \text{as } |x| \to \infty.
	$$
\end{proof}

\begin{proof}[Proof of Theorem \ref{thm:viscosity}]

We shall verify the hypotheses of the Perron method (Lemma \ref{lem:perron}).

First, by Proposition \ref{prop:global-subsolution}, there exists $c_* > 0$ such that for every $c > c_*$, there is a viscosity subsolution $\underline{u}_c \in C(\overline{\Omega})$ of the special Lagrangian equation satisfying
$$
\underline{u}_c = \varphi \quad \text{on } \partial \Omega, \qquad 
\underline{u}_c \le Q_c \quad \text{in } \mathbb{R}^n \setminus \overline{\Omega},
$$
and
\begin{equation}\label{Q}
\underline{u}_c(x) = Q_c(x) + O(|x|^{2-n}) \quad \text{as } |x| \to \infty. 
\end{equation}
Second, the quadratic polynomial
$$
\overline{u}(x) := Q_c(x) = \frac{1}{2}x^T A x + b\cdot x + c
$$
is a classical solution of the special Lagrangian equation, since $D^2\overline{u}=A$ and $A\in\mathcal{A}_\Theta$. Thus $\overline{u}$ is a viscosity supersolution. Increasing $c$ if necessary, we have $\overline{u} \ge \underline{u}_c$ in $\Omega$ by $\underline{u}_c \le Q_c$ and $\overline{u}=Q_c$.
Moreover, from \eqref{Q},
$$
\lim_{|x|\to\infty} (\overline{u} - \underline{u}_c)(x) = 0.
$$
Thus all the hypotheses of Lemma \ref{lem:perron} are satisfied with $ \mathbb{R}^n \setminus \overline{\Omega}$. Applying Lemma \ref{lem:perron}, we obtain a unique viscosity solution $u \in C^0(\overline{\Omega})$ of the special Lagrangian equation such that
$$
\underline{u}_c \le u \le \overline{u} \quad \text{in } \mathbb{R}^n \setminus \overline{\Omega}, \qquad u = \varphi \quad \text{on } \partial \Omega,
$$
and
$$
u(x) = Q_c(x) + O(|x|^{2-n}) \quad \text{as } |x| \to \infty.
$$

By the Perron method (Lemma \ref{lem:perron}) and the comparison principle (Lemma \ref{lem:comparison}), the function $u$ is continuous, solves the special Lagrangian equation in the viscosity sense, and satisfies the prescribed quadratic asymptotics 
$$
u(x) = Q_c(x) + O(|x|^{2-n}) \quad \text{as } |x| \to \infty.
$$
Uniqueness follows from the comparison principle.
\end{proof}

\section{Smooth strict global subsolutions}

\begin{lemma}
	\label{lem:local-subsolution}
	Let \(\Omega\subset\mathbb R^n\) be a smooth bounded strictly convex domain, and let 
	$\varphi\in C^\infty(\partial \Omega)$. Let $U\subset\mathbb R^n$ be a smooth bounded 
	domain such that $\overline \Omega\subset U$, and let $\Phi\in C^\infty(\overline U\setminus \Omega)$
	be a smooth extension of $\varphi$, namely
	$$
	\Phi=\varphi \quad \text{on } \partial \Omega.
	$$
	Let $d(x):=\operatorname{dist}(x,\partial \Omega)$ be the  distance function, 
	positive in \(\Omega\). Then, for every phase
	$$
	\frac{(n-2)\pi}{2}\leq\Theta<\frac{n\pi}{2},
	$$
	there exists a sufficiently large constant \(N>1\), depending on \(n,\Theta,\Omega,U,\Phi\), such that the function
	\begin{equation}
		v_0(x):=\Phi(x)+(d(x)+1)^N-1, \qquad x\in \overline U\setminus \Omega,
		\label{eq:v0-li-xiao}
	\end{equation}
	satisfies
	\begin{enumerate}
		\item[(i)] $v_0=\varphi$ on $\partial \Omega$;
		\item[(ii)] $\displaystyle F(D^2v_0):=\sum_{i=1}^n \arctan \lambda_i(D^2v_0) > \Theta$ 
		in $\overline U\setminus \overline \Omega$.
	\end{enumerate}
\end{lemma}

\begin{proof}
	Take $U$ to be a sufficiently small tubular neighborhood of $\partial \Omega$ so that 
	$d$ is smooth in $\overline U\setminus \Omega$. Let $\xi\in\partial \Omega$ and choose an 
	orthonormal frame $\{e_1,\dots,e_{n-1},\nu\}$ at $\xi$, where $\nu$ is the 
	exterior unit normal and $\{e_a\}$ are the principal directions. Let $\kappa_a>0$
	be the principal curvatures with respect to the interior normal. On $\partial \Omega$,
	$$
	d=0,\qquad \nabla d=\nu,\qquad D^2d(e_a,e_a)=\kappa_a,\qquad D^2d(\nu,\nu)=0.
	$$
	A direct computation gives
	\begin{equation}
		D^2v_0\big|_{\partial \Omega}
		=
		\begin{pmatrix}
			\Phi_{ab}+N\kappa_a\delta_{ab} & \Phi_{an}\\
			\Phi_{na} & \Phi_{nn}+N(N-1)
		\end{pmatrix}.
		\label{eq:v0-hessian}
	\end{equation}
	Hence the eigenvalues of $D^2v_0\big|_{\partial \Omega}$ satisfy
	$$
	\lambda_a=N\kappa_a+O(1),\quad a=1,\dots,n-1,\qquad
	\lambda_n=N(N-1)+O(1).
	$$
	Since $\kappa_a>0$, all eigenvalues tend to $+\infty$ as $N\to\infty$. Therefore
	$$
	F(D^2v_0)\big|_{\partial \Omega}
	=\sum_{a=1}^{n-1}\arctan(N\kappa_a+O(1))+\arctan(N(N-1)+O(1))
	\longrightarrow \frac{n\pi}{2}.
	$$
	Since $\Theta<n\pi/2$, choose $N$ large so that $F(D^2v_0)>\Theta$ on $\partial \Omega$. By continuity, after possibly shrinking $U$, 
	the strict inequality holds throughout $\overline U\setminus\overline \Omega$. 
	Finally, 
	$v_0=\varphi$ on $\partial \Omega$ since $d=0$ and $\Phi=\varphi$ there.
\end{proof}

\begin{lemma}\label{brg}
	Let $a > 0$ and $C_b \in \mathbb{R}$. Define
	$$
	u_{\mathrm{mid}}(x) := Q_0(x) + a\rho(x) + C_b.
	$$
	Then $u_{\mathrm{mid}}$ satisfies
	$$
	\sum_{i=1}^n \arctan \lambda_i(D^2 u_{\mathrm{mid}}) \ge \Theta.
	$$
\end{lemma}

\begin{proof}
	Since \(u_{\mathrm{mid}}(x)=Q_0(x)+a\rho(x)+C_b\), we have
	$$
	D^2u_{\mathrm{mid}}(x)=A+aD^2\rho(x).
	$$
	Let $y=G^{-1/2}x$. Then $\rho(x)=|y|$, and a direct computation gives
	$$
	D_y^2\rho=\frac{I}{|y|}-\frac{y\otimes y}{|y|^3}.
	$$
	Its eigenvalues are
	$$
	\lambda_1=\cdots=\lambda_{n-1}=\frac{1}{|y|}>0,\qquad \lambda_n=0,
	$$
	so \(D_y^2\rho\ge0\). Returning to the \(x\)-coordinates,
	$$
	D_x^2\rho=G^{-1/2}\left(\frac{I}{|y|}-\frac{y\otimes y}{|y|^3}\right)G^{-1/2}\ge0.
	$$
	Hence \(D^2u_{\mathrm{mid}}(x)\ge A\). Since \(F(M)=\sum_i\arctan\lambda_i(M)\) is monotone, it follows that
	$$
	F(D^2u_{\mathrm{mid}})\ge F(A)=\Theta.
	$$
	Moreover, if \(D^2\rho\neq0\), the inequality is strict, since \(a>0\).
\end{proof}

We now glue the local subsolution \(v_0\) constructed in Lemma \eqref{lem:local-subsolution} to the bridge 
function \(v_{\mathrm{br}}\) defined in Lemma \eqref{brg}. Choose radii \(0<r_1<r_2\) such that
$$
\overline \Omega \subset E_{r_1}, \qquad E_{r_2} \subset U,
$$
where \(U\) is the tubular neighborhood on which \(v_0\) is defined.

Define
$$
m_1 := \min_{\partial E_{r_1}} (v_0 - Q_0),
\qquad
M_2 := \max_{\partial E_{r_2}} (v_0 - Q_0).
$$
Fix \(\mu>0\) and choose \(a>0\) so large that
$$
a(r_2 - r_1) > M_2 - m_1 + 6\mu.
$$
Set
$$
C_b := m_1 - 2\mu - ar_1.
$$
Then on \(\partial E_{r_1}\),
$$
v_{\mathrm{br}} - Q_0 = ar_1 + C_b = m_1 - 2\mu \le v_0 - Q_0 - 2\mu,
$$
and on \(\partial E_{r_2}\),
$$
v_{\mathrm{br}} - Q_0 = ar_2 + C_b = m_1 - 2\mu + a(r_2 - r_1) > M_2 + 4\mu \ge v_0 - Q_0 + 4\mu.
$$
These inequalities are strict with a positive margin. Hence, by continuity, the same 
inequalities hold in neighborhoods of the two interfaces within \(E_{r_2}\setminus E_{r_1}\). 
That is, there exist neighborhoods \(\mathcal N_1\) of \(\partial E_{r_1}\) and 
\(\mathcal N_2\) of \(\partial E_{r_2}\), both contained in \(E_{r_2}\setminus E_{r_1}\), 
such that
$$
v_0 \ge v_{\mathrm{br}} + \mu \quad \text{in } \mathcal N_1,
$$
and
$$
v_{\mathrm{br}} \ge v_0 + \mu \quad \text{in } \mathcal N_2.
$$
Let \(M_\mu\) be the regularized maximum from Lemma \eqref{max-sub}. Define
$$
v_1(x) :=
\begin{cases}
	v_0(x), & x \in E_{r_1} \setminus \overline \Omega, \\[1.2ex]
	M_\mu(v_0(x), v_{\mathrm{br}}(x)), & x \in E_{r_2} \setminus E_{r_1}, \\[1.2ex]
	v_{\mathrm{br}}(x), & x \in \mathbb R^n \setminus E_{r_2}.
\end{cases}
$$
By Lemma \eqref{max-sub}, \(v_1\in C^\infty(\mathbb R^n\setminus \Omega)\) is a smooth subsolution. 
Near \(\partial E_{r_1}\), the regularized maximum equals \(v_0\), and near 
\(\partial E_{r_2}\), it equals \(v_{\mathrm{br}}\), so no loss of smoothness occurs 
at the interfaces. Moreover, \(v_1=\varphi\) on \(\partial \Omega\).

\begin{lemma}
	\label{far-gluing}
	There exist constants \(\tau_0>0\), \(\Lambda>0\), and \(c_0>0\) such that 
	for every \(0<\tau\le\tau_0\) and every \(c\ge c_0\), setting
	$$
	R_c:=\tau c,\qquad \delta_c:=\Lambda c^2 R_c^{n-2},
	$$
	there exists \(B_c>0\) satisfying
	$$
	\frac12 cR_c^{n-2}\le B_c\le CcR_c^{n-2},
	$$
	such that
	$$
	u_\infty^c(x):=Q_c(x)-B_c\rho(x)^{2-n}+\delta_c\rho(x)^{-n}
	$$
	is a smooth far-field subsolution of the special Lagrangian equation in 
	\(\mathbb R^n\setminus E_{R_c}\), and satisfies the interface gluing conditions
	\begin{equation}
		u_\infty^c \le v_{\mathrm{br}} - 2\mu \quad \text{on } \partial E_{R_c},
		\label{eq:inner-interface}
	\end{equation}
	\begin{equation}
		u_\infty^c \ge v_{\mathrm{br}} + 2\mu \quad \text{on } \partial E_{2R_c}.
		\label{eq:outer-interface}
	\end{equation}
\end{lemma}

\begin{proof}
	On \(\partial E_{R_c}\), we have \(\rho=R_c\). Since 
	\(v_{\mathrm{br}}=Q_0+a\rho+C_b\), the condition 
	\(u_\infty^c\le v_{\mathrm{br}}-2\mu\) is equivalent to
	\begin{equation}
		B_c \ge \left(c+\delta_cR_c^{-n}-aR_c-C_b+2\mu\right)R_c^{n-2}.
		\label{eq:theta-lower}
	\end{equation}
	Since \(R_c=\tau c\) and \(\delta_cR_c^{-n}=\Lambda\tau^{-2}\), the right-hand side 
	of \eqref{eq:theta-lower} equals
	
	$$
	\left[(1-a\tau)c+\Lambda\tau^{-2}-C_b+2\mu\right]R_c^{n-2}.
	$$
	
	Choosing \(\tau>0\) so small that \(a\tau\le 1/4\), we obtain
	\begin{equation}
		\left(c+\delta_cR_c^{-n}-aR_c-C_b+2\mu\right)R_c^{n-2}
		\ge \frac12 cR_c^{n-2}
		\label{eq:theta-lower-estimate}
	\end{equation}
	for all sufficiently large \(c\).

	On \(\partial E_{2R_c}\), we have \(\rho=2R_c\). The condition 
	\(u_\infty^c\ge v_{\mathrm{br}}+2\mu\) is equivalent to
	\begin{equation}
		B_c \le \left(c+\delta_c(2R_c)^{-n}-2aR_c-C_b-2\mu\right)(2R_c)^{n-2}.
		\label{eq:theta-upper}
	\end{equation}
	Since \(\delta_c(2R_c)^{-n}=\Lambda\tau^{-2}2^{-n}\), the right-hand side of 
	\eqref{eq:theta-upper} equals
	$$
	\left[(1-2a\tau)c+\Lambda\tau^{-2}2^{-n}-C_b-2\mu\right](2R_c)^{n-2}.
	$$
	Thus, for some constant \(C>0\) independent of \(c\),
	\begin{equation}
		\left[(1-2a\tau)c+\Lambda\tau^{-2}2^{-n}-C_b-2\mu\right](2R_c)^{n-2}
		\le CcR_c^{n-2}.
		\label{eq:theta-upper-estimate}
	\end{equation}
	Combining \eqref{eq:theta-lower-estimate} and \eqref{eq:theta-upper-estimate}, 
	if we choose \(B_c\) in the interval
	\begin{equation}
		\left[
		\left(c+\delta_cR_c^{-n}-aR_c-C_b+2\mu\right)R_c^{n-2},\;
		\left(c+\delta_c(2R_c)^{-n}-2aR_c-C_b-2\mu\right)(2R_c)^{n-2}
		\right],
		\label{eq:theta-interval}
	\end{equation}
	then both interface conditions hold.
	
	It remains to verify that the interval \eqref{eq:theta-interval} is nonempty. 
	Indeed, the difference between the upper and lower endpoints is
	\begin{equation}
		\begin{aligned}
			&\left(c+\delta_c(2R_c)^{-n}-2aR_c-C_b-2\mu\right)(2R_c)^{n-2} \\
			&\quad - \left(c+\delta_cR_c^{-n}-aR_c-C_b+2\mu\right)R_c^{n-2} \\
			&= R_c^{n-2}\left[(2^{n-2}-1)c + O(1) + O(R_c)\right],
		\end{aligned}
		\label{eq:interval-difference}
	\end{equation}
	which is positive for \(c\) sufficiently large, since \(2^{n-2}-1>0\) for \(n\ge3\). 
	Therefore, the interval is nonempty for all sufficiently large \(c\).
	
	Finally, by Lemma \ref{lem:farfield}, \(u_\infty^c\) is a smooth subsolution 
	of the special Lagrangian equation in \(\mathbb R^n\setminus E_{R_c}\) for all 
	sufficiently large \(c\), provided \(B_c\) and \(\delta_c\) satisfy the 
	smallness conditions in that lemma. This completes the proof.
\end{proof}

We now define the second gluing function
\begin{equation}
	v_2(x) :=
	\begin{cases}
		v_{\mathrm{br}}(x), & x \in E_{R_c}\setminus \overline \Omega, \\[1.2ex]
		M_\mu(v_{\mathrm{br}}(x), u_\infty^c(x)), & x \in E_{2R_c}\setminus E_{R_c}, \\[1.2ex]
		u_\infty^c(x), & x \in \mathbb R^n\setminus E_{2R_c}.
	\end{cases}
	\label{eq:v2}
\end{equation}
Finally,   define 
\begin{equation}
	\underline u_c(x) :=
	\begin{cases}
		v_1(x), & x \in E_{R_c}\setminus \overline \Omega, \\[1.2ex]
		v_2(x), & x \in E_{2R_c}\setminus E_{R_c}, \\[1.2ex]
		u_\infty^c(x), & x \in \mathbb R^n\setminus E_{2R_c}.
	\end{cases}
	\label{eq:underline-u}
\end{equation}
 By Lemmas \ref{lem:local-subsolution}, \ref{brg} and \ref{far-gluing}, 
  \(\underline u_c\) is a smooth subsolution of 
 the special Lagrangian equation in \(\Omega\), i.e.,
 \begin{equation}
 	F(D^2\underline u_c)\ge \Theta \quad \text{in } \mathbb{R}^n \setminus \overline{\Omega}.
 \end{equation}
 Moreover, \(\underline u_c=\varphi\) on \(\partial \Omega\).

It remains to verify that the global subsolution constructed above lies below the 
quadratic polynomial \(Q_c\). We check this in each region.

First, the near subsolution \(v_0\) and the first gluing region are contained 
in the fixed bounded set \(E_{r_2}\setminus \Omega\). Hence there exists a constant 
\(C_0>0\), independent of the large parameter \(c\), such that
\begin{equation}
	v_0-Q_0\le C_0 \quad \text{in } E_{r_2}\setminus \Omega.
\end{equation}
Taking \(c>C_0+2\mu\), we obtain
\begin{equation}
	v_0\le Q_c-2\mu \quad \text{in } E_{r_2}\setminus \Omega.
\end{equation}
Second, for the bridge function \(v_{\mathrm{br}}=Q_0+a\rho+C_b\), we compute
\begin{equation}
	v_{\mathrm{br}}-Q_c = a\rho+C_b-c.
\end{equation}
Since \(\rho\le 2R_c=2\tau c\) on the bridge region, the choice \(2a\tau\le 1/2\) and 
a further increase of \(c\) give
\begin{equation}
	v_{\mathrm{br}}\le Q_c-2\mu \quad \text{for } \rho\le 2R_c.
\end{equation}
Third, in the far-field region, Lemma \ref{lem:farfield} gives
\begin{equation}
	u_\infty^c\le Q_c \quad \text{in } \mathbb R^n\setminus E_{R_c}.
\end{equation}

Finally, by Lemma \ref{lem:regmax}, the regularized maximum increases the ordinary 
maximum by at most \(\mu\). In each gluing region, the two relevant branches have 
been arranged to lie below \(Q_c-2\mu\). Therefore
\begin{equation}
	\underline u_c\le Q_c \quad \text{in } \Omega.
\end{equation}

\section{The  bounded  problems}\label{sectionbp}
Let \(S > 2R_c\) be arbitrary and set
$\Omega_S := E_S \setminus\overline \Omega.$
On \(\Omega_S\) we consider the following bounded Dirichlet problem:
\begin{equation}\label{bounded}
	\begin{cases}
		\displaystyle \sum_{i=1}^n \arctan \lambda_i(D^2u_S) = \Theta, & x\in \Omega_S, \\[1.2ex]
		u_S=\varphi, & x\in \partial \Omega, \\[1.2ex]
		u_S=u_\infty^c, & x\in \partial E_S.
	\end{cases}
\end{equation}

	For the Dirichlet problem on annular domains,
	the estimates on the inner boundary \(\partial \Omega\) are naturally independent of 
	\(S\), since the geometry of \(\partial \Omega\) is fixed. For the outer boundary 
	\(\partial E_S\), however, the standard approach, as in \cite{BJ26} Lemma 5.4 and 
	\cite{LX26} Section 3.3.2, is to apply a scaling transformation that fixes the 
	outer boundary. The resulting boundary estimates then follow from the 
	CNS-Trudinger--Guan \cite{CNS85,T95,G14} theory for the Dirichlet problem 
	on bounded domains.

For the special Lagrangian equation, the Dirichlet problem on bounded domains in 
the supercritical regime \(\Theta>(n-2)\pi/2\) was solved by Collins--Picard--Wu 
\cite{CPW17}. However, the critical case \(\Theta=(n-2)\pi/2\) remains an open problem. 

In this section, we solve the critical bounded-domain Dirichlet problem and, 
by applying the annular approximation scheme of Bao--Jiang \cite{BJ26} and 
Li--Xiao \cite{LX26}, extend the result to the exterior domain.

\subsection{Critical bounded problem}

The Dirichlet problem for the critical special Lagrangian equation
\begin{equation}\label{eq1}
	\left\{
	\begin{aligned}
		\sum_{i=1}^n \arctan \lambda_i(D^2 u) &= \frac{(n-2)\pi}{2} \quad &&\text{in } \Omega,\\
		u &= \varphi \quad &&\text{on } \partial\Omega,
	\end{aligned}
	\right.
\end{equation}

	\begin{theorem}\label{thm:main}
	Let $\Omega \subset \mathbb{R}^n$ be a bounded domain of class $C^\infty$, let $\varphi \in C^\infty(\partial\Omega)$.
	Assume that there exists a subsolution $\underline{u} \in C^\infty(\overline{\Omega})$ satisfying
	\begin{equation}\label{sub}
		\left\{
		\begin{aligned}
			\sum_{i=1}^n \arctan \lambda_i(D^2\underline{u}) &\geq \frac{(n-2)\pi}{2} \quad &&\text{in } \Omega,\\
			\underline{u} &= \varphi \quad &&\text{on } \partial\Omega.
		\end{aligned}
		\right.
	\end{equation}
	Then the Dirichlet problem \eqref{eq1} admits a unique solution $u \in C^\infty(\overline{\Omega})$. 
\end{theorem}

\begin{theorem}\label{thm:uniform}
	Assume the same conditions as in Theorem \ref{thm:main}. 
	For each sufficiently small $t>0$, let $u^t \in  C^{\infty}(\overline{\Omega})$ 
	be the unique solution of the approximating Dirichlet problem
	\begin{equation}\label{approx}
		\left\{
		\begin{aligned}
			\sum_{i=1}^n \arctan \lambda_i(D^2 u^t + tI) &= (n-2)\frac{\pi}{2} + C_0 t \quad &&\text{in } \Omega,\\
			u^t &= \varphi \quad &&\text{on } \partial\Omega.
		\end{aligned}
		\right.
	\end{equation}
	Then there exist constants $C_0>0$ and $t_0>0$, depending only on 
	$\Omega$, $\varphi$ and $\underline u$, but independent of $t$, 
	such that for all $0<t<t_0$,
	\begin{equation}\label{uniform_estimate}
		\|u^t\|_{C^{2}(\overline{\Omega})} \le C.
	\end{equation}
\end{theorem}

\begin{lemma}\label{lem:subsolution}
	Let $\underline\lambda_1 \ge \underline\lambda_2 \ge \cdots \ge \underline\lambda_n$ be the eigenvalues of $D^2\underline u$, where $\underline u \in C^\infty(\overline{\Omega})$ satisfies the subsolution condition
	\begin{equation}
		\sum_{i=1}^n \arctan \underline\lambda_i(x) \ge (n-2)\frac{\pi}{2} \quad \text{for all } x \in \overline{\Omega}.
	\end{equation}
	Then there exists a constant $A_0 > (n-3)\frac{\pi}{2}$ such that
	\begin{equation}
		\sum_{i=2}^n \arctan \underline\lambda_i(x) > A_0 \quad \text{for all } x \in \overline{\Omega}.
	\end{equation}
	Equivalently,
	\begin{equation}
		\min_{1\le j\le n} \sum_{i\ne j} \arctan \underline\lambda_i(x) > A_0 \quad \text{for all } x \in \overline{\Omega}.
	\end{equation}
\end{lemma}

\begin{proof}
	Since $\underline u \in C^\infty(\overline{\Omega})$, the eigenvalues $\underline\lambda_i(x)$ are continuous and finite on the compact set $\overline{\Omega}$. In particular,
	\begin{equation}
		\arctan \underline\lambda_1(x) < \frac{\pi}{2} \quad \text{for all } x \in \overline{\Omega}.
	\end{equation}
	Combining this with the subsolution condition yields
	\begin{equation}
		\sum_{i=2}^n \arctan \underline\lambda_i(x)
		= \sum_{i=1}^n \arctan \underline\lambda_i(x) - \arctan \underline\lambda_1(x)
		> (n-2)\frac{\pi}{2} - \frac{\pi}{2}
		= (n-3)\frac{\pi}{2}
	\end{equation}
	for all $x \in \overline{\Omega}$.
	
	Define
	\begin{equation}
		H(x) := \sum_{i=2}^n \arctan \underline\lambda_i(x) - (n-3)\frac{\pi}{2}.
	\end{equation}
	Then $H$ is continuous on $\overline{\Omega}$ and $H(x) > 0$ for all $x \in \overline{\Omega}$. Since $\overline{\Omega}$ is compact, there exists $a_0 > 0$ such that
	\begin{equation}
		H(x) \ge a_0 \quad \text{for all } x \in \overline{\Omega}.
	\end{equation}
	Taking $A_0 := (n-3)\frac{\pi}{2} + \frac{a_0}{2}$ .
	
	The equivalence follows from the monotonicity of $\arctan$: since $\underline\lambda_1$ is the largest eigenvalue, the minimum over $j$ of $\sum_{i\ne j} \arctan \underline\lambda_i$ is attained at $j=1$, i.e.
	\begin{equation}
		\min_{1\le j\le n} \sum_{i\ne j} \arctan \underline\lambda_i(x)
		= \sum_{i=2}^n \arctan \underline\lambda_i(x).
	\end{equation}
\end{proof}

\begin{lemma}\label{lem:subsolution_preserved}
	Assume that $\underline u \in C^\infty(\overline{\Omega})$ satisfies the subsolution condition
	\begin{equation}
		\sum_{i=1}^n \arctan \underline\lambda_i(x) \ge (n-2)\frac{\pi}{2} \quad \text{for all } x \in \overline{\Omega},
	\end{equation}
	where $\underline\lambda_1 \ge \cdots \ge \underline\lambda_n$ are the eigenvalues of $D^2\underline u$. Define
	\begin{equation}
		C_0 := \inf_{\substack{x\in \overline{\Omega}\\ s_i\in[0,1]}} \sum_{i=1}^n \frac{1}{1+(\underline\lambda_i(x)+s_i)^2} > 0.
	\end{equation}
	For $0<t<1$, set
	\begin{equation}
		\Theta_t := (n-2)\frac{\pi}{2} + C_0 t,
	\end{equation}
	and consider the approximating problem
	\begin{equation}
		\sum_{i=1}^n \arctan \lambda_i(D^2u^t + tI) = \Theta_t.
	\end{equation}
	Then, for all $0<t<1$,
	\begin{equation}
		\sum_{i=1}^n \arctan \lambda_i(D^2\underline u + tI) \ge \Theta_t \quad \text{in } \Omega.
	\end{equation}
	In particular, $\underline u$ is a subsolution of the approximating problem for every $0<t<1$.
\end{lemma}

\begin{proof}
	Since $D^2\underline u + tI$ has eigenvalues $\underline\lambda_i + t$, by the mean value theorem,
	\begin{equation}
		\sum_{i=1}^n \arctan(\underline\lambda_i + t) - \sum_{i=1}^n \arctan \underline\lambda_i
		= t \sum_{i=1}^n \frac{1}{1+\xi_i^2}
	\end{equation}
	for some $\xi_i \in (\underline\lambda_i, \underline\lambda_i+t) \subset (\underline\lambda_i, \underline\lambda_i+1)$, since $0<t<1$. 
	
	Since $\underline u \in C^\infty(\overline{\Omega})$, the eigenvalues $\underline\lambda_i$ are bounded on $\overline{\Omega}$. Hence
	$$
	\sum_{i=1}^n \frac{1}{1+\xi_i^2} \ge C_0 > 0,
	$$
	by the definition of $C_0$. Therefore,
	\begin{equation}
		\sum_{i=1}^n \arctan(\underline\lambda_i + t)
		\ge \sum_{i=1}^n \arctan \underline\lambda_i + C_0 t
		\ge (n-2)\frac{\pi}{2} + C_0 t
		= \Theta_t.
	\end{equation}
	This proves the lemma.
\end{proof}

\begin{lemma}[\cite{CPW17} Lemma 2.1]\label{lem:yuan}
	Assume that
	\begin{equation}
		\sum_{i=1}^n \arctan \lambda_i \ge \tau
	\end{equation}
	for some $\tau \in [(n-2)\frac{\pi}{2}, n\frac{\pi}{2})$, and that the components of 
	$\lambda = (\lambda_1,\ldots,\lambda_n)$ are ordered so that 
	$\lambda_1 \ge \lambda_2 \ge \cdots \ge \lambda_n$. Then the following hold:
	\begin{enumerate}
		\item[(i)] $\lambda_1 \ge \lambda_2 \ge \cdots \ge \lambda_{n-1} > 0$ and $\lambda_{n-1} \ge |\lambda_n|$;
		\item[(ii)] $\lambda_1 + (n-1)\lambda_n \ge 0$;
		\item[(iii)] the level set 
		$$
		\Gamma^\tau := \left\{\lambda \in \mathbb{R}^n : \sum_{i=1}^n \arctan \lambda_i > \tau \right\}
		$$
		is convex;
		\item[(iv)] $\displaystyle \sum_{i=1}^n \lambda_i \ge 0$.
	\end{enumerate}
\end{lemma}

\subsection{$C^{0}$ estimate and Gradient estimate}

\begin{theorem}\label{thm:C0}
	Assume the same conditions as in Theorem \ref{thm:main}. 
	For each $0<t<1$, let $u^t \in C^{\infty}(\overline{\Omega})$ 
	be the unique solution of the approximating Dirichlet problem \eqref{approx}.
	Then there exists a constant $C>0$, depending only on $\Omega$, $\varphi$, and $\underline u$, 
	but independent of $t$, such that for all $0<t<1$,
	\begin{equation}
		\|u^t\|_{C^0(\overline{\Omega})} \le C.
	\end{equation}
\end{theorem}

\begin{proof}
	For the lower bound, by Lemma \ref{lem:subsolution_preserved}, $\underline u$ is a subsolution of \eqref{approx} for every $0<t<1$. Since the operator is elliptic and $\underline u = u^t = \varphi$ on $\partial\Omega$, the comparison principle yields
	\begin{equation}
		u^t \ge \underline u \quad \text{in } \Omega.
	\end{equation}
	
	For the upper bound, let $\mu_1 \ge \cdots \ge \mu_n$ be the eigenvalues of $D^2u^t$. 
	From \eqref{approx}, the eigenvalues of $D^2u^t + tI$ are $\mu_i + t$ and satisfy
	$$
	\sum_{i=1}^n \arctan(\mu_i + t) = (n-2)\frac{\pi}{2} + C_0 t \ge (n-2)\frac{\pi}{2}.
	$$
	By Lemma \ref{lem:yuan} (iv), we have
	$$
	\sum_{i=1}^n (\mu_i + t) \ge 0,
	$$
	i.e.
	$$
	\Delta u^t + nt \ge 0 \quad \text{in } \Omega.
	$$
	
	Let $w \in C^2(\overline{\Omega})$ be the harmonic function satisfying
	$$
	\Delta w = 0 \quad \text{in } \Omega, \qquad w = \varphi + \frac{t}{2}|x|^2 \quad \text{on } \partial\Omega.
	$$
	Define
	$$
	v := u^t - w + \frac{t}{2}|x|^2.
	$$
	Then
	$$
	\Delta v = \Delta u^t + nt \ge 0 \quad \text{in } \Omega,
	$$
	and
	$$
	v = 0 \quad \text{on } \partial\Omega.
	$$
	By the maximum principle,
	$$
	v = u^t - w + \frac{t}{2}|x|^2 \le 0 \quad \text{in } \Omega,
	$$
	hence
	$$
	u^t \le w - \frac{t}{2}|x|^2 \quad \text{in } \Omega.
	$$
	
	Since $\Omega$ is bounded, there exists $R>0$ such that $|x| \le R$ for all $x \in \overline{\Omega}$. By the maximum principle for the harmonic function $w$, we have
	$$
	w(x) \le \max_{\partial\Omega} \left(\varphi + \frac{t}{2}|x|^2\right) \le \max_{\partial\Omega} \varphi + \frac{R^2}{2}.
	$$
	Therefore,
	$$
	u^t(x) \le w(x) - \frac{t}{2}|x|^2 \le \max_{\partial\Omega} \varphi + \frac{R^2}{2}.
	$$
	Combining the lower and upper bounds, we obtain
	$$
	\|u^t\|_{C^0(\overline{\Omega})} \le C,
	$$
	where $C$ is independent of $t$.
\end{proof}

\begin{lemma}\label{lem:linearized}
	Let $u^t$ be a solution of the approximating problem \eqref{approx}. Then for $0<t<1$:
	
	(i) The linearized operator $L$ of \eqref{approx} takes the form
	$$
	L(v) = \sum_{i,j} \mathcal{F}^{ij} v_{ij},
	$$
	where
	$$
	\mathcal{F}^{ij} = \frac{\partial}{\partial M_{ij}}\left(\sum_{p=1}^n \arctan \lambda_p(M)\right), \qquad M = D^2u^t + tI.
	$$
	
	(ii) Differentiating the approximating equation \eqref{approx} with respect to $x_k$ gives
$$
	\sum_{i,j} \mathcal{F}^{ij} u^t_{kij} = 0.
$$
	
	(iii) Differentiating twice gives
	$$
	\sum_{i,j} \mathcal{F}^{ij} u^t_{kkij} + \sum_{i,j,p,q} \mathcal{F}^{ij,pq} u^t_{kij} u^t_{kpq} = 0,
	$$
	where
	$$
	\mathcal{F}^{ij,pq} = \frac{\partial^2}{\partial M_{ij}\partial M_{pq}}\left(\sum_{p=1}^n \arctan \lambda_p(M)\right).
	$$
	
	(iv) Let $Q(M) = -e^{-k\mathcal{F}(M)}$ with $k$ sufficiently large. 
	Then $Q$ is concave. Since $Q(D^2u^t + tI) =-e^{-k\Theta(t)} $ is constant, differentiating gives
	$$
	Q^{ij} u^t_{kij} = 0,
	$$
	and differentiating twice gives
	$$
	Q^{ij} u^t_{kkij} + Q^{ij,pq} u^t_{kij} u^t_{kpq} = 0,
$$
	where
	$$
	Q^{ij} = ke^{-k\mathcal{F}}F^{ij}, \qquad 
	Q^{ij,pq} = ke^{-k\mathcal{F}}\left(\mathcal{F}^{ij,pq} - k\mathcal{F}^{ij}F^{pq}\right).
	$$
	Since $Q$ is concave, we have $Q^{ij,pq}u^t_{kij}u^t_{kpq} \le 0$.
\end{lemma}

\begin{theorem}\label{thm:gradient}
	Assume the same conditions as in Theorem \ref{thm:main}. 
	For each sufficiently small $t>0$, let $u^t \in C^{\infty}(\overline{\Omega})$ 
	be the unique solution of the approximating Dirichlet problem \eqref{approx}.
	Then there exists a constant $C>0$, depending only on $\Omega$, $\varphi$, and $\underline u$, 
	but independent of $t$, such that for all $0<t<1$,
	\begin{equation}
		\|\nabla u^t\|_{C^0(\overline{\Omega})} \le C.
	\end{equation}
\end{theorem}

\begin{proof}
	We first prove a uniform boundary gradient bound. From the $C^0$ estimate (Theorem \ref{thm:C0}) we have
	$$
	\underline u \le u^t \le w - \frac{t}{2}|x|^2 \quad \text{in } \Omega,
	$$
	and
	$$
	\underline u = u^t = w - \frac{t}{2}|x|^2 = \varphi \quad \text{on } \partial\Omega.
	$$
	Hence, 
	$$
	D_\nu \underline u \le D_\nu u^t \le D_\nu\left(w - \frac{t}{2}|x|^2\right) \quad \text{on } \partial\Omega,
	$$
	which implies
	$$
	\max_{\partial\Omega} |D u^t| \le C,
	$$
	where $C$ is independent of $t$.
	
Now let $v = |\nabla u^t|^2$. By Lemma \ref{lem:linearized} (ii), differentiating the approximating equation gives
$$
\mathcal{F}^{ij} u^t_{kij} = 0.
$$
Let $L = \mathcal{F}^{ij}D_{ij}$ be the linearized operator. Then
$$
L v = 2\sum_{k=1}^n \mathcal{F}^{ij} u^t_{ki} u^t_{kj} \ge 0,
$$
since $\mathcal{F}^{ij}$ is positive definite. By the maximum principle,
$$
\max_{\overline{\Omega}} v \le \max_{\partial\Omega} v.
$$
Combining this with the boundary gradient bound, we conclude
$$
\|\nabla u^t\|_{C^0(\overline{\Omega})} \le C,
$$
where $C$ is independent of $t$.
	
\end{proof}

\subsection{Second Order Estimates}
\begin{lemma} [\cite{FYZ26}, Lemma 3.2]\label{lem:dichotomy} 
	Let $\mu \in \mathbb{R}^n$ satisfy the subsolution condition
	\begin{equation}
		A(\mu) := \min_{1\le j\le n}\sum_{i\ne j}\arctan \mu_i > A_0 > (n-3)\frac{\pi}{2}.
	\end{equation}
	Then there exist constants $\delta_0>0$ and $t_0>0$, depending only on $\mu$ and $A_0$, 
	such that for any $0<t<t_0$ and any $\lambda \in \mathbb{R}^n$ satisfying
	\begin{equation}
		\sum_{i=1}^n \arctan \lambda_i = (n-2)\frac{\pi}{2} + C_0 t,
	\end{equation}
	one of the following holds:
	\begin{equation}
		\sum_{i=1}^n \frac{\mu_i-\lambda_i}{1+\lambda_i^2} \ge \delta_0 \sum_{i=1}^n \frac{1}{1+\lambda_i^2},
	\end{equation}
	or
	\begin{equation}
		\frac{1}{1+\lambda_j^2} \ge \delta_0 \sum_{i=1}^n \frac{1}{1+\lambda_i^2} \quad \text{for all } j=1,\ldots,n.
	\end{equation}
\end{lemma}	

	\begin{lemma}[\cite{FYZ26}, Lemma 3.3] \label{lem:linearized_dichotomy}
	Let $\underline u$ be a subsolution of the critical equation and let $u^t$ be the solution of the approximating problem \eqref{approx}. 
	Then there exist uniform constants $\delta_0>0$ and $t_0>0$ such that for any $0<t<t_0$ and any $x\in\Omega$, 
	one of the following holds at $x$:
	\begin{equation}
		\sum_{i,j} \mathcal{F}^{ij}\left(\underline u_{ij} - u^t_{ij}\right) \ge \delta_0 \sum_{i,j} \mathcal{F}^{ij} \delta_{ij},
	\end{equation}
	or
	\begin{equation}
		\mathcal{F}_{min} \ge \delta_0 \sum_{i,j} \mathcal{F}^{ij} \delta_{ij},
	\end{equation}
	where $\mathcal{F}_{min}$ denotes the smallest eigenvalue of the matrix $(\mathcal{F}^{ij})$.
\end{lemma}

\begin{theorem}\label{prop:interior_hessian}
	Assume the same conditions as in Theorem \ref{thm:main}. 
	For each $0<t<1$, let $u^t \in C^{\infty}(\overline{\Omega})$ 
	be the unique solution of the approximating Dirichlet problem \eqref{approx}.
	Then there exists a constant $C>0$, independent of $t$, such that
	\begin{equation}
		\sup_{\Omega} |D^2u^t| \le C\left(1 + \max_{\partial\Omega} |D^2u^t|\right).
	\end{equation}
\end{theorem}

\begin{proof}
	Fix $k \in \{1,\dots,n\}$, differentiating the equation twice with respect to $x_k$ gives
	$$
	Q^{ij}u^t_{kkij} + Q^{ij,rs}u^t_{kij}u^t_{krs} = 0.
	$$
	Since $Q$ is concave, we have $Q^{ij,rs}u^t_{kij}u^t_{krs} \le 0$. Hence
	$$
	Q^{ij}\partial_{ij}(D_{kk}u^t) \ge 0 \quad \text{in } \Omega.
	$$
	By Lemma \ref{lem:linearized} (iv), the matrix $[Q^{ij}]$ is positive definite, so the maximum principle applies to $D_{kk}u^t$. Therefore,
	$$
	D_{kk}u^t(x) \le \max_{\partial\Omega} D_{kk}u^t \le \max_{\partial\Omega} |D^2u^t|, \quad \forall x \in \Omega. 
	$$
	Now fix $x_0 \in \Omega$ and rotate coordinates so that $D^2u^t(x_0)$ is diagonal with eigenvalues $\mu_1 \ge \cdots \ge \mu_n$. For each $k=1,\dots,n$ gives
	$$
	\mu_i(x_0) \le \max_{\partial\Omega} |D^2u^t|, \quad i=1,\dots,n. 
	$$
	On the other hand, by Lemma \ref{lem:yuan} (iv), since 
	$\sum_{i=1}^n \arctan(\mu_i + t) \ge (n-2)\pi/2$, we have
	$$
	\sum_{i=1}^n (\mu_i + t) \ge 0,
	$$
	i.e.
	$$
	\Delta u^t(x_0) + nt \ge 0.
	$$
	Since $0<t<1$, this implies $\Delta u^t(x_0) \ge -n$. Thus
	$$
	\mu_n(x_0) = \Delta u^t(x_0) - \sum_{i=1}^{n-1} \mu_i(x_0)
	\ge -n - (n-1)\max_{\partial\Omega} |D^2u^t|. 
	$$
	We obtain for all $i=1,\dots,n$,
	$$
	|\mu_i(x_0)| \le C\left(1 + \max_{\partial\Omega} |D^2u^t|\right).
	$$
	Since orthogonal transformations preserve the Frobenius norm,
	$$
	|D^2u^t(x_0)| = \left(\sum_{i=1}^n \mu_i(x_0)^2\right)^{1/2}
	\le C\left(1 + \max_{\partial\Omega} |D^2u^t|\right).
	$$
	Since $x_0 \in \Omega$ is arbitrary, the desired estimate follows.
\end{proof}

At any point $x_0 \in \partial\Omega$, choose coordinates $x_1,\ldots,x_n$ with origin at $x_0$
such that the positive $x_n$-axis is in the direction of the interior normal of
$\partial\Omega$ at $0$. Denote $x' = (x_1,\ldots,x_{n-1})$. Near $0$, we may represent $\partial\Omega$ as
a graph
\begin{equation}
	x_n = \rho(x').
\end{equation}
Since $u^t = \underline u = \varphi$ on $\partial\Omega$, we have
$$
(u^t - \underline u)(x', \rho(x')) = 0.
$$
Differentiating this identity twice with respect to $x_\alpha$, $x_\beta$ ($\alpha,\beta < n$) and evaluating at $0$ gives
\begin{equation}
	u^t_{\alpha\beta}(0) = \underline u_{\alpha\beta}(0) - (u^t_n(0)-\underline u_{n}(0))\rho_{\alpha\beta}(0), \qquad \alpha,\beta < n. 
\end{equation}
By the boundary gradient estimate (Theorem \ref{thm:gradient}), we have
$$
|u^t_n(0)| \le C,
$$
where $C$ is independent of $t$. Since $\rho \in C^\infty(\partial\Omega)$, the quantities $\rho_{\alpha\beta}(0)$ are bounded. Also, $\underline u \in C^\infty(\overline{\Omega})$, so $\underline u_{\alpha\beta}(0)$ is bounded. Therefore, 
\begin{equation}
	|u^t_{\alpha\beta}(0)| \le C, \qquad \alpha,\beta < n, 
\end{equation}
where $C$ is independent of $t$.

\begin{lemma}\label{lem:boundary_barrier}
	There exist uniform constants $\rho_0>0$, $\epsilon_0>0$, $N>0$ such that for all $0<t<1$, the function
$$
	v := (u^t - \underline u) + \epsilon_0 d - \frac{N}{2}d^2
	$$
	satisfies $v \ge 0$ on $\overline{\Omega}_{\rho_0}$ and
$$
	\mathcal{F}^{ij}\partial_{ij}v \le -\delta_0\sum_i F^{ii} \quad \text{in } \Omega_{\rho_0},
	$$
	where $\Omega_{\rho_0} := \{x \in \Omega : d(x) < \rho_0\}$.
\end{lemma}

\begin{proof}
	Let $d(x) := \operatorname{dist}(x,\partial\Omega)$ be the distance function, which is smooth in a sufficiently small neighborhood of $\partial\Omega$. $\rho_0\leq2\epsilon_0 /N$ after $\epsilon_0, N$ being fixed.
	
	We compute
	$$
	\mathcal F^{ij}\partial_{ij}v
	= \mathcal F^{ij}\partial_{ij}(u^t-\underline u)
	+ (\epsilon_0 - Nd)\mathcal F^{ij}\partial_{ij}d
	- N \mathcal F^{ij}\partial_i d\partial_j d. 
$$
\textbf{Case 1:} $|\lambda| \le R$. 
In this case, there exist uniform constants $0 < c_1 \le \mathcal F^{ij} \le C_1$, 
so $\mathcal F^{ij}\nabla_i d\nabla_j d \ge c_1$. Taking $N$ sufficiently large 
yields the desired inequality.

	\noindent\textbf{Case 2:} $|\lambda| > R$. 
	By Lemma \ref{lem:linearized_dichotomy}, one of the following holds:

	\noindent\textbf{Case 2a:} 
	$$
	\mathcal F^{ij}\nabla_{ij}(u-\underline u) \le -\delta_0 \sum_i \mathcal F^{ii}.
	$$
	Using $|\mathcal F^{ij}\nabla_{ij}d| \le C\sum_i \mathcal F^{ii}$, we have
	$$
	\begin{aligned}
		\mathcal F^{ij}\nabla_{ij}v
		&= \mathcal F^{ij}\nabla_{ij}(u-\underline u)
		+ (\epsilon_0 - Nd)\mathcal F^{ij}\nabla_{ij}d
		- N\mathcal F^{ij}\nabla_i d\nabla_j d \\
		&\le -\delta_0 \sum_i \mathcal F^{ii}
		+ C(\epsilon_0 + Nd)\sum_i \mathcal F^{ii}
		- N\mathcal F^{ij}\nabla_i d\nabla_j d \\
		&\le -\delta_0 \sum_i \mathcal F^{ii}
		+ C(\epsilon_0 + Nd)\sum_i \mathcal F^{ii}.
	\end{aligned}
	$$
	Choose $\epsilon_0>0$  sufficiently small  and  $N$  large such that
	$$
	C(\epsilon_0 + Nd) \le \frac{\delta_0}{2}.
	$$
	Then
	$$
	\mathcal F^{ij}\nabla_{ij}v \le -\frac{\delta_0}{2}\sum_i \mathcal F^{ii}.
	$$
	\noindent\textbf{Case 2b:} 
	$Q$ is concave. 
	Since $u^t$ is a solution and $\underline u$ is a subsolution of the approximating problem, 
	we have $Q(D^2u^t+tI) \le Q(D^2\underline u+tI)$. By concavity of $Q$,
$$
	0<Q(D^2\underline u+tI)-Q(D^2u^t+tI)\leq Q^{ij}(\underline u-u^t)_{ij}.
	$$
	Since $Q^{ij}=ke^{-kF}\mathcal F^{ij}$ and $ke^{-kF}>0$, we obtain
$$
	\mathcal F^{ij}\partial_{ij}(u^t-\underline u) \le 0.
	$$
	There exists $\delta_0>0$ such that
$$
	\mathcal F^{ii} \ge \delta_0 \sum_k \mathcal F^{kk} \quad \text{for all } i.
$$
	Since $|\nabla d|=1$,
$$
	\mathcal F^{ij}\partial_i d\partial_j d = \sum_i \mathcal F^{ii}(\partial_i d)^2 \ge \delta_0 \sum_k \mathcal F^{kk}.
$$
	Also, $|\mathcal F^{ij}\partial_{ij}d| \le C\sum_i \mathcal F^{ii}$. Hence, 
	$$
	\begin{aligned}
		\mathcal F^{ij}\partial_{ij}v
		&= \mathcal F^{ij}\partial_{ij}(u^t-\underline u) + (\epsilon_0 - Nd)\mathcal F^{ij}\partial_{ij}d - N \mathcal F^{ij}\partial_i d\partial_j d \\
		&\le 0 + C(\epsilon_0 + Nd)\sum_i \mathcal F^{ii} - N\delta_0 \sum_i \mathcal F^{ii}.
	\end{aligned}
	$$
	Taking $N$ sufficiently large and $\epsilon_0$ sufficiently small yields
$$
	\mathcal F^{ij}\partial_{ij}v \le - \delta_0 \sum_i \mathcal F^{ii}.
	$$
\end{proof}
	\begin{lemma}\label{lem:mixed_boundary}
		Under the assumptions of Theorem \ref{thm:main}, there exists a constant $C>0$,
		independent of $t$, such that for all $0<t<t_0$,
		\begin{equation}
			|D_{\alpha n}u^t|_{L^\infty(\partial\Omega)} \le C, \qquad 1\le \alpha \le n-1.
		\end{equation}
	\end{lemma}
	
	\begin{proof}
		Fix a point $x_0 \in \partial\Omega$ and choose local coordinates with $x_0=0$,
		$x_n$ the interior normal, and $x_\alpha$ ($\alpha<n$) tangential directions.
		Near $0$, $\partial\Omega$ is given by $x_n=\rho(x')$, with $\rho(0)=0$, $\nabla_{x'}\rho(0)=0$.
		
		For $1\le \alpha \le n-1$, define the approximate tangential operator
	$$
		T_\alpha := \partial_\alpha + \sum_{\beta=1}^{n-1} \rho_{\alpha\beta}(0)
		\left(x_\beta\partial_n - x_n\partial_\beta\right).
	$$
		This operator agrees with the true tangential derivative $\partial_\alpha + \rho_\alpha\partial_n$ 
		on $\partial\Omega$ up to second order:
	$$
		T_\alpha = \partial_\alpha + \rho_\alpha\partial_n + O(|x'|^2)\partial_n 
		- \sum_\beta \rho_{\alpha\beta}(0)\rho(x')\partial_\beta.
	$$
		
		Since $u^t = \underline u = \varphi$ on $\partial\Omega$, we have
		\begin{equation}
			|T_\alpha(u^t - \underline u)| \le C|x'|^2 \quad \text{on } \partial\Omega. 
		\end{equation}
		
		Also, since $T_\alpha$ is an infinitesimal rotation up to lower-order terms and 
		$F$ is rotationally invariant, applying $T_\alpha$ to the approximating equation 
		$F(D^2u^t+tI)=\Theta_t$ yields
	$$
		L(T_\alpha u^t) = 0.
	$$
		Using the subsolution property of $\underline u$, we obtain
		\begin{equation}
			|L(T_\alpha(u^t - \underline u))| \le C\sum_i \mathcal F^{ii} \quad \text{in } \Omega. 
		\end{equation}
		
		Let $v$ be the barrier function from Lemma \ref{lem:boundary_barrier}, and set
		$\Lambda := \Omega \cap B_\mu(0)$ with $\mu>0$ sufficiently small.
		
		Choose constants $A>0$ and $B>0$, independent of $t$, such that
	$$
		L\left(A v + B|x|^2 \pm T_\alpha(u^t - \underline u)\right) \le 0 \quad \text{in } \Lambda,
	$$
		and
	$$
		A v + B|x|^2 \pm T_\alpha(u^t - \underline u) \ge 0 \quad \text{on } \partial\Lambda.
	$$

		By the maximum principle,
	$$
		A v + B|x|^2 \pm T_\alpha(u^t - \underline u) \ge 0 \quad \text{in } \Lambda.
	$$
		Since equality holds at $x=0$, we have
	$$
		\partial_n\left(A v + B|x|^2 \pm T_\alpha(u^t - \underline u)\right)(0) \ge 0.
	$$
		Using the definition of $v$ and the boundary gradient estimate (Theorem \ref{thm:gradient}), 
		we obtain
	$$
		|D_{\alpha n}u^t(0) - D_{\alpha n}\underline u(0)| \le C.
	$$
		Hence
$$
		|D_{\alpha n}u^t(0)| \le C.
	$$
		Since $0\in\partial\Omega$ is arbitrary, the desired estimate follows.
	\end{proof}
	\subsection{Boundary normal estimate}
	
	We first establish the lower bound. Since \(\Delta u^t \ge -n\) (by Lemma \ref{lem:yuan} (iv) and \(0<t<1\)) and \(|D_{\alpha\beta}u^t|\le C\) for \(\alpha,\beta<n\), we have
	\begin{equation}
		u^t_{nn} = \Delta u^t - \sum_{\alpha=1}^{n-1} u^t_{\alpha\alpha} \ge -n - (n-1)C \ge -C.
	\end{equation}
	Thus
	\begin{equation}
		u^t_{nn} \ge -C \quad \text{on } \partial\Omega. \label{eq:unn_lower}
	\end{equation}
	
	Now we establish the upper bound. Let \(\lambda'(u^t_{\alpha\beta})\) denote the eigenvalues of the \((n-1)\times(n-1)\) matrix \(u^t_{\alpha\beta}\).
	
	Let \(x_0 \in \partial\Omega\) be a point where the function
$$
	x \mapsto \sum_{\alpha=1}^{n-1} \arctan \lambda'_\alpha(D^2u^t(x))
$$
	attains its minimum on \(\partial\Omega\).
	
	Define the reduced operator, for a fixed large constant \(R>0\),
$$
	\tilde F(M) := \sum_{\alpha=1}^{n-1} \arctan \lambda_\alpha(M) + \arctan R.
$$
	Define
$$
	\tilde F_0^{\alpha\beta} := \frac{\partial \tilde F}{\partial M_{\alpha\beta}}(u^t_{\alpha\beta}(x_0)).
$$
	
	\begin{lemma}\label{lem:unn_relation}
		For any \(0<t<1\) and any \(x\in\partial\Omega\),
$$
		\sum_{\alpha=1}^{n-1} \arctan \lambda'_\alpha(u^t_{\alpha\beta}(x))
		\ge
		\sum_{i=1}^n \arctan \lambda_i(D^2u^t(x))
		-
		\arctan u^t_{nn}(x).
$$
	\end{lemma}
	
	\begin{proof}
		Fix \(x\in\partial\Omega\). Using an orthogonal transformation on the tangential directions, we may diagonalize the \((n-1)\times(n-1)\) upper block of \(D^2u^t(x)\), yielding
$$
		\begin{pmatrix}
			\lambda_1' & & & * \\
			& \ddots & & \vdots \\
			& & \lambda_{n-1}' & * \\
			* & \cdots & * & u_{nn}
		\end{pmatrix}.
$$
		By the Schur-Horn theorem, the diagonal vector
$$
		(\lambda_1',\ldots,\lambda_{n-1}',u_{nn})
$$
		is in the convex hull of the vectors that are permutations of the eigenvalues
$$
		(\lambda_1,\ldots,\lambda_n)
$$
		of \(D^2u^t(x)\). Since the level set
$$
		\Gamma := \left\{ \mu \in \mathbb{R}^n : \sum_{i=1}^n \arctan \mu_i \ge \Theta_t \right\}
$$
		is convex for \(\Theta_t = (n-2)\pi/2 + C_0 t \ge (n-2)\pi/2\) (Lemma \ref{lem:yuan}), any convex combination of permutations of \((\lambda_1,\ldots,\lambda_n)\) lies in \(\Gamma\). Therefore,
$$
		(\lambda_1',\ldots,\lambda_{n-1}',u^t_{nn}) \in \Gamma.
$$
		Thus
$$
		\sum_{\alpha=1}^{n-1} \arctan \lambda'_\alpha + \arctan u^t_{nn}
		\ge \Theta_t
		= \sum_{i=1}^n \arctan \lambda_i.
$$
		This proves the lemma.
	\end{proof}
	
	\begin{lemma}\label{lem:unn_tau}
		For any \(0<t<1\), if \(u^t_{nn}(x_0)\) is sufficiently large, then there exists \(\tau>0\), independent of \(t\), such that
$$
		\tilde F_0^{\alpha\beta}\left(\underline u_{\alpha\beta}(x_0)-u^t_{\alpha\beta}(x_0)\right) \ge \tau.
$$
	\end{lemma}
	
	\begin{proof}
		Applying Lemma \ref{lem:unn_relation} to the subsolution \(\underline u\), we have
		$$
		\sum_{\alpha=1}^{n-1} \arctan \underline\lambda'_\alpha
		\ge
		\sum_{i=1}^n \arctan \underline\lambda_i
		-
		\arctan \underline u_{nn}.
		$$
		
		By the Caffarelli-Nirenberg-Spruck lemma, as \(u^t_{nn}\to\infty\),
	$$
		\lambda_\alpha = \lambda'_\alpha + o(1) \quad (1\le \alpha \le n-1), \qquad
		\lambda_n = u^t_{nn} + O(1).
		$$
		Hence, for any \(\gamma>0\), if \(u^t_{nn}\) is sufficiently large,
	$$
		\sum_{\alpha=1}^{n-1} \arctan \lambda'_\alpha
		\le
		\sum_{\alpha=1}^{n-1} \arctan \lambda_\alpha
		+ \frac{\gamma}{2}.
		$$
		Therefore,
		\begin{align*}
			\sum_{\alpha=1}^{n-1} \arctan \underline\lambda'_\alpha
			-
			\sum_{\alpha=1}^{n-1} \arctan \lambda'_\alpha
			&\ge
			\sum_{i=1}^n \arctan \underline\lambda_i
			-
			\arctan \underline u_{nn}
			-
			\sum_{\alpha=1}^{n-1} \arctan \lambda_\alpha
			-
			\frac{\gamma}{2} \\
			&\ge
			\Theta_t
			-
			\arctan \underline u_{nn}
			-
			\left(\Theta_t - \arctan \lambda_n\right)
			-
			\frac{\gamma}{2} \\
			&\ge
			\arctan u^t_{nn}
			-
			\arctan \underline u_{nn}
			-
			\gamma.
		\end{align*}
		Since \(u^t_{nn}\) is large, there exists \(c>0\), independent of \(t\), such that
		$$
		\sum_{\alpha=1}^{n-1} \arctan \underline\lambda'_\alpha
		\ge
		\sum_{\alpha=1}^{n-1} \arctan \lambda'_\alpha
		+ c.
	$$
		Thus \(\underline u_{\alpha\beta}(x_0)\) lies strictly inside the level set
	$$
		\Gamma := \left\{ M \in \mathrm{Sym}(n-1) : \tilde F(M) \ge \tilde F(u^t_{\alpha\beta}(x_0)) \right\},
	$$
		while \(u^t_{\alpha\beta}(x_0)\) lies on its boundary. Since \(\tilde F_0^{\alpha\beta}\) is the interior normal at \(u^t_{\alpha\beta}(x_0)\), and the level set is strictly convex, we obtain
	$$
		\tilde F_0^{\alpha\beta}\left(\underline u_{\alpha\beta}(x_0)-u^t_{\alpha\beta}(x_0)\right) \ge \tau > 0.
	$$
		This completes the proof.
	\end{proof}

	Now define \(\eta := \tilde F_0^{\alpha\beta}\sigma_{\alpha\beta}(x_0)\), where \(\sigma_{\alpha\beta}\) is the second fundamental form of \(\partial\Omega\). Since
	$$
	u^t_{\alpha\beta}-\underline u_{\alpha\beta} = -(u^t-\underline u)_n \sigma_{\alpha\beta} \quad \text{on } \partial\Omega,
$$
	we have
$$
	\underline u_{\alpha\beta}(x_0)-u^t_{\alpha\beta}(x_0) = (u^t-\underline u)_n(x_0)\sigma_{\alpha\beta}(x_0).
$$
	Contracting with \(\tilde F_0^{\alpha\beta}\) gives
$$
	\tilde F_0^{\alpha\beta}\left(\underline u_{\alpha\beta}(x_0)-u^t_{\alpha\beta}(x_0)\right)
	=
	\nabla_n(u^t-\underline u)(x_0)\,\eta.
$$
	By Lemma \ref{lem:unn_tau}, \(\tau \le \nabla_n(u^t-\underline u)(x_0)\eta \le C\eta\), so
$$
	\eta \ge \frac{\tau}{C} > 0.
$$
	Define in \(\Lambda := B_r(x_0)\cap \Omega\)
$$
	\Phi := -\nabla_n(u^t-\underline u)
	+ \frac{1}{\eta}\tilde F_0^{\alpha\beta}\left(\underline u_{\alpha\beta}(x)-u^t_{\alpha\beta}(x_0)\right).
$$
	 On \(\partial\Omega\cap B_r(x_0)\),
$$
	\Phi
	= \frac{1}{\eta}\tilde F_0^{\alpha\beta}\left(u^t_{\alpha\beta}(x)-u^t_{\alpha\beta}(x_0)\right) \ge 0,
$$
	since \(x_0\) is the minimizer of \(\tilde F\) on \(\partial\Omega\).
	In the interior, differentiating \(\Phi\) and using the linearized operator \(L=\mathcal F^{ij}D_{ij}\), we have
$$
	|L\Phi| \le C\sum_i \mathcal F^{ii}.
$$
	Let \(v\) be the barrier function from Lemma \ref{lem:boundary_barrier}. Choose \(A,B>0\) independent of \(t\) such that
$$
	L(Av+B|x|^2+\Phi) \le 0 \quad \text{in } \Lambda,
$$
	and
$$
	Av+B|x|^2+\Phi \ge 0 \quad \text{on } \partial\Lambda.
$$
	By the maximum principle,
$$
	Av+B|x|^2+\Phi \ge 0 \quad \text{in } \Lambda.
$$
	Since equality holds at \(x=0\),
$$
	\partial_n(Av+B|x|^2+\Phi)(0) \ge 0.
$$
	Using the definition of \(v\), this gives
$$
	\Phi_n(0) \ge -C,
$$
	and hence
$$
	u^t_{nn}(x_0) \le C. \label{eq:unn_upper_x0}
$$
	
	Now we propagate this bound to all boundary points. By Lemma \ref{lem:unn_relation} applied at \(x_0\),
	$$
	\sum_{\alpha=1}^{n-1} \arctan \lambda'_\alpha(x_0)
	\ge \Theta_t - \arctan u^t_{nn}(x_0)
	\ge \Theta_t - \arctan C.
	$$
	Since \(x_0\) is a minimum point of \(\sum_{\alpha=1}^{n-1} \arctan \lambda'_\alpha\) on \(\partial\Omega\), for any \(x\in\partial\Omega\),
	$$
	\sum_{\alpha=1}^{n-1} \arctan \lambda'_\alpha(x)
	\ge
	\sum_{\alpha=1}^{n-1} \arctan \lambda'_\alpha(x_0)
	\ge \Theta_t - \arctan C.
	$$
	Adding \(\arctan R\) and taking \(R\) sufficiently large such that \(\arctan R - \arctan C \ge c_0 > 0\), we obtain
	\begin{equation}
		\sum_{\alpha=1}^{n-1} \arctan \lambda'_\alpha(x) + \arctan R
		\ge \Theta_t + c_0 \quad \text{for all } x\in\partial\Omega. \label{eq:prop4}
	\end{equation}
	
	Finally, suppose there exists \(y\in\partial\Omega\) such that \(u^t_{nn}(y) > R_1\) for some large \(R_1 \ge R\). then
	$$
	\sum_{i=1}^n \arctan \lambda_i(y)
	\ge
	\sum_{\alpha=1}^{n-1} \arctan \lambda'_\alpha(y) + \arctan u^t_{nn}(y) - \frac{c_0}{2}.
	$$
	
	Using \eqref{eq:prop4} with \(R=R_1\) and \(u^t_{nn}(y) > R_1\),
	
	$$\sum_{\alpha=1}^{n-1} \arctan \lambda'_\alpha(y) + \arctan u^t_{nn}(y)
	\ge \Theta_t + c_0.
	$$
	Thus
$$
	\sum_{i=1}^n \arctan \lambda_i(y) \ge \Theta_t + \frac{c_0}{2},
	$$
	
	which contradicts the fact that \(u^t\) solves the approximating equation with phase \(\Theta_t\). Therefore,
	
	$$
	u^t_{nn}(x) \le R_1 \quad \text{for all } x\in\partial\Omega.
	$$
	
	Combining with the lower bound \eqref{eq:unn_lower}, we obtain
	\begin{equation}
		|u^t_{nn}|_{L^\infty(\partial\Omega)} \le C. \label{eq:unn_boundary}
	\end{equation}
	Combining \eqref{eq:unn_boundary} with the tangential estimate and the mixed estimate (Lemma \ref{lem:mixed_boundary}), we have
	\begin{equation}
		|D^2u^t|_{L^\infty(\partial\Omega)} \le C. \label{eq:hessian_boundary}
	\end{equation}
	Finally, by Theorem \ref{prop:interior_hessian},
$$
	\sup_{\Omega} |D^2u^t| \le C\left(1 + \max_{\partial\Omega} |D^2u^t|\right) \le C.
$$
	Combining this with the \(C^0\) estimate (Theorem \ref{thm:C0}) and the gradient estimate (Theorem \ref{thm:gradient}), we obtain the uniform \(C^2\) estimate
$$
	\|u^t\|_{C^2(\overline{\Omega})} \le C,
$$
	where \(C\) is independent of \(t\).

\section{The annular problem}

\begin{lemma}\label{lem:inner_boundary}
	Let $u_S$ be the solution of the annular problem \eqref{bounded}. 
	There exists a constant $C>0$, independent of $S$, such that
	\begin{equation}
		|Du_S| + |D^2u_S| \le C \quad \text{on } \partial \Omega.
	\end{equation}
\end{lemma}
\begin{proof}
	The argument is local near $\partial \Omega$ and is independent of the outer radius $S$. In a fixed bounded neighborhood of $\partial \Omega$, the global smooth subsolution $\underline u_c$ agrees with the near-boundary subsolution $v_0$ (Lemma \ref{lem:local-subsolution}). Therefore, the boundary gradient estimate and the tangential, mixed, and normal second derivative estimates are precisely those established in Section \ref{sectionbp} and \cite{CPW17}.
\end{proof}

\begin{lemma}\label{lem:outer_boundary}
	Let $u_S$ be the solution of the annular problem \eqref{bounded}. 
	There exists a constant $C>0$, independent of $S$, such that
	\begin{equation}
		|D^2u_S| \le C \quad \text{on } \partial E_S, 
	\end{equation}
	and
	\begin{equation}
		|Du_S| \le C S \quad \text{on } \partial E_S. 
	\end{equation}
\end{lemma}

\begin{proof}
	The estimate is reduced to a fixed outer boundary by scaling. Set
	\begin{equation}
		U_S(y) := S^{-2}u_S(Sy), \qquad y \in E_1 \setminus S^{-1}\overline \Omega.
	\end{equation}
	Then
	$$
	D_y U_S = S^{-1}D_x u_S, \qquad D_y^2 U_S = D_x^2 u_S,
	$$
	and $U_S$ satisfies the same Hessian equation in $E_1 \setminus S^{-1}\Omega$. 
	For $S$ sufficiently large, the scaled domain $S^{-1}\Omega$ lies strictly inside $E_1$. 
	Hence a collar neighborhood of the outer boundary $\partial E_1$, with width 
	independent of $S$, is contained in $E_1 \setminus S^{-1}\Omega$. The boundary 
	estimates near $\partial E_1$ therefore depend only on the fixed local geometry 
	of $\partial E_1$ and the scaled boundary data, not on the inner boundary.
	
	Near $\partial E_1$, the boundary value is
	$$
	g_S(y) := S^{-2}u_\infty^c(Sy)
	= \frac{1}{2}y^T A y + S^{-1}b \cdot y + S^{-2}c
	- B_c S^{-n} \rho(y)^{2-n} + \delta_c S^{-n-2} \rho(y)^{-n}.
	$$
	For $S$ sufficiently large, $g_S$ converges uniformly in $C^2$ to 
	$\frac{1}{2}y^T A y$ on $\partial E_1$. Therefore the boundary estimates 
	for $U_S$ are uniform in $S$, and hence
	$$
	|D_y^2 U_S| \le C \quad \text{on } \partial E_1,
	$$
	which gives $|D^2u_S| \le C$ on $\partial E_S$. 
	
	For the gradient estimate, since $D_y U_S = S^{-1}D_x u_S$, we obtain
$$
	|Du_S| = S|D_y U_S| \le C S \quad \text{on } \partial E_S.
$$
	This completes the proof.
\end{proof}

The inner boundary estimates (Lemma \ref{lem:inner_boundary}) and the outer boundary estimates (Lemma \ref{lem:outer_boundary}) give
$$
|D^2u_S| \le C \quad \text{on } \partial \Omega_S.
$$
Combining this with the interior Hessian estimate (similar to Theorem \ref{prop:interior_hessian}), we obtain the uniform global second derivative bound.

\begin{proposition}\label{prop:uniform_C2_annular}
	There exists a constant $C>0$, independent of $S$, such that
	\begin{equation}
		|D^2u_S| \le C \quad \text{in } \Omega_S. 
	\end{equation}
\end{proposition}

\begin{proposition}\label{prop:large_constant_exterior}
	Under the hypotheses of Theorem \ref{thm:smooth}, there exists a constant $c_0 \in \mathbb{R}$ such that for every $c > c_0$, the exterior Dirichlet problem \eqref{eqsle} admits a unique smooth solution $u \in C^\infty(\overline{\Omega})$ satisfying the quadratic asymptotic \eqref{2-n}.
\end{proposition}

\begin{proof}
	Let $c$ be so large that the smooth global subsolution $\underline u_c$ exists.
	Let $S_j \to \infty$ with $S_j > 2R_c$. By the Evans--Krylov theorem and Schauder estimates, together with the inner boundary estimates (Lemma \ref{lem:inner_boundary}) and the outer boundary estimates (Lemma \ref{lem:outer_boundary}), there exist a subsequence, still denoted $u_{S_j}$, and a function
	$$
	u \in C^\infty(\mathbb{R}^n\setminus\Omega)
	$$
	such that $u_{S_j} \to u$ in $C^m(K)$ for every compact set $K \Subset \mathbb{R}^n\setminus\Omega$ and every $m \ge 0$. Passing to the limit gives
	$$
	\sum_{i=1}^n \arctan \lambda_i(D^2u) = \Theta \quad \text{in } \mathbb{R}^n\setminus\Omega,
$$
	and
	$$
	u = \varphi \quad \text{on } \partial \Omega.
	$$
	Moreover, the inequality $\underline u_c \le u_S \le Q_c$ passes to the limit:
	$$
	\underline u_c \le u \le Q_c \quad \text{in }\mathbb{R}^n\setminus\Omega.
	$$
	Since $\underline u_c = u_\infty^c$ for $\rho \ge 2R_c$, we obtain for large $|x|$,
$$
	0 \le Q_c(x) - u(x) \le Q_c(x) - u_\infty^c(x)
	= B_c \rho(x)^{2-n} - \delta_c \rho(x)^{-n}.
	$$
	Because $\rho \sim |x|$, this gives
$$
	u(x) = Q_c(x) + O(|x|^{2-n}), \qquad |x| \to \infty.
$$
	Existence follows from the limiting argument above. Uniqueness follows from the comparison principle (Lemma \ref{lem:comparison}).
\end{proof}

\section{Applications to quadratic Hessian equations in dimension three}

We shall use the following algebraic form of the special Lagrangian equation: 
\begin{equation}
	\label{eq:algebraic-SLE}
	\cos\Theta
	\sum_{0\le 2k+1\le n}
	(-1)^k \sigma_{2k+1}(\lambda(D^2u))
	-
	\sin\Theta
	\sum_{0\le 2k\le n}
	(-1)^k \sigma_{2k}(\lambda(D^2u))
	=
	0.
\end{equation}
In dimension $3$, this becomes
$$
\cos\Theta
(\sigma_1-\sigma_3)
-
\sin\Theta
(1-\sigma_2)
=
0.
$$
Therefore, if $\Theta=\pi/2$, then
$$
\sigma_2=1,
$$
while if $\Theta=\pi$, then
$$
\sigma_1=\sigma_3.
$$
In dimension $4$, the algebraic form is
$$
\cos\Theta
(\sigma_1-\sigma_3)
-
\sin\Theta
(1-\sigma_2+\sigma_4)
=
0.
$$
Thus, if $\Theta=\pi$, then
$$
\sigma_3=\sigma_1.
$$

	\begin{corollary}
		\label{cor:3d-sigma2}
		Let $\Omega\subset\mathbb R^3$ be a bounded uniformly convex $C^{1,1}$ domain, and let
		$\varphi\in C^0(\partial \Omega)$ be uniformly semiconvex with respect to $\partial \Omega$.
		Let $A\in\mathrm{Sym}(3)$ satisfy
		$$
		\sigma_2(\lambda(A))=1,
		$$
		and let $b\in\mathbb R^3$. Then there exists $c_*>0$ such that for every
		$c>c_*$, the exterior Dirichlet problem
		$$
		\begin{cases}
			\sigma_2(D^2u)=1, & x\in\mathbb{R}^n\setminus\overline\Omega,\\[1ex]
			u=\varphi, & x\in\partial \Omega,
		\end{cases}
		$$
		admits a unique viscosity solution $u\in C^0(\mathbb{R}^n\setminus\Omega)$ satisfying
		$$
		u(x)=\frac12 x^{T}Ax+b\cdot x+c+O(|x|^{-1}),\qquad |x|\to\infty.
		$$
		If $\Omega$ is smooth and strictly convex and
		$\varphi\in C^{\infty}(\partial \Omega)$, then the solution $u\in C^{\infty} (\mathbb{R}^n\setminus\Omega)$
	.
	\end{corollary}

\begin{corollary}
	\label{cor:3d-bounded-sigma2}
	Let $\Omega\subset\mathbb R^3$ be a bounded smooth domain, and let
	$\varphi\in C^\infty(\partial\Omega)$. Assume that there exists a smooth
	subsolution $\underline u\in C^\infty(\overline\Omega)$ satisfying
$$
	\begin{cases}
		\sigma_2(D^2\underline u)\ge 1, & x\in\Omega,\\[1ex]
		\underline u=\varphi, & x\in\partial\Omega.
	\end{cases}
$$
	Then the Dirichlet problem
$$
	\begin{cases}
		\sigma_2(D^2u)=1, & x\in\Omega,\\[1ex]
		u=\varphi, & x\in\partial\Omega,
	\end{cases}
$$
	admits a unique smooth solution $u\in C^\infty(\overline\Omega)$.
\end{corollary}

	\subsection*{Acknowledgments}
This work was partially supported by NSFC $12171389$ and NSFC $11801015$.

\end{document}